\documentclass[11pt]{amsart}
\usepackage[a4paper]{geometry}                                              
\usepackage{fourier,mathtools}                                              
\usepackage[bb=ams, cal=cm, scr=boondox, frak=euler]{mathalpha}             
\let\amsmathbb\mathbb
\AtBeginDocument{%
    \let\mathbb\relax
    \newcommand{\mathbb}[1]{\amsmathbb{#1}}
}
\mathtoolsset{mathic=true}

\usepackage{amsmath,amsthm,amsfonts,amssymb}                                
\usepackage{mathrsfs}                                                       
\usepackage{esint}                                                          
\usepackage[all]{xy}                                                        
\xyoption{rotate}
\usepackage{tikz-cd}                                                        
\usepackage[colorlinks=true,linkcolor=magenta,citecolor=blue]{hyperref}     
\usepackage[capitalise]{cleveref}                                           
\usepackage{colonequals}                                                    

\newcommand{\bQ}{{\mathbb{Q}}} \newcommand{\bR}{{\mathbb{R}}}

 \newcommand{\bZ}{{\mathbb{Z}}}

\DeclareMathOperator{\Stab}{Stab}

\DeclareMathOperator{\Tr}{Tr}

\newcommand{\tors}{{\mathrm{tors}}}
\newcommand{\tr}{{\mathrm{tr}}}

\newcommand{\Gm}{\mathbb{G}_m}
\newcommand{\hh}{\widehat h}

\usepackage{enumitem}

\theoremstyle{plain}

\newtheorem{theorem}{Theorem}[section]           
\newtheorem{corollary}[theorem]{Corollary}
\newtheorem{lemma}[theorem]{Lemma}
\newtheorem{proposition}[theorem]{Proposition}

\theoremstyle{definition}

\newtheorem{example}[theorem]{Example}

\title{Mordell--Lang plus Geometric Bogomolov}
\date{29 September 2026}
\author{Long Liu}
\address{Institute for Theoretical Sciences, Westlake University, Hangzhou, Zhejiang, 310030, China}
\email{liulong@westlake.edu.cn}    

\begin{document}
\begin{abstract}
    We prove a geometric version of Poonen's "Mordell--Lang plus Bogomolov" theorem for semi-abelian varieties in characteristic zero. This is a generalization of the Mordell--Lang conjecture and the geometric Bogomolov conjecture.
\end{abstract}

\maketitle

\setcounter{tocdepth}{1}
\tableofcontents
\raggedbottom

\section{Introduction}

We prove a geometric version of Poonen's "Mordell--Lang plus Bogomolov" theorem for semi-abelian varieties in characteristic zero. This is a generalization of the Mordell--Lang conjecture and the geometric Bogomolov conjecture. More precisely, we study subvarieties of a semi-abelian variety over a function field that contain a Zariski dense set of points arbitrarily close in canonical height to a finite-rank division group. We characterize them, after passage to the stabilizer quotient, as images of constant subvarieties translated by points of the group. A geometric Vojta inequality supplies the height bound needed for this classification.

\subsection{Semi-abelian varieties and heights}

Let $k$ be an algebraically closed field of characteristic zero, and let $K/k$ be a finitely generated regular extension with $\operatorname{trdeg}(K/k)\geq1$. Fix an algebraic closure $\overline K$. Fix a polarization $(\mathcal B,\mathcal H)$ of $K/k$, where $\mathcal B$ is a normal projective $k$-variety with function field $K$ and $\mathcal H$ is an ample line bundle on $\mathcal B$. Heights are normalized by degrees over $K$, so they are unchanged when the field of definition of a point is enlarged.

Let $G$ be a semi-abelian variety over $K$, given by an exact sequence
\[
    1\longrightarrow T\longrightarrow G\xrightarrow{\pi}A\longrightarrow0,
\]
where $T$ is a torus and $A$ is an abelian variety. We use additive notation for the group law and write $[n]$ for multiplication by $n\in\bZ$.

We use the height conventions of Luo--Yu \cite[\S1.1 and \S2.5]{LY2025Bogomolov}. Put $t=\dim T$. After a finite extension of $K$, choose a splitting $T\simeq\Gm^t$ and the associated standard equivariant compactification $\overline G$. The projection extends to $\overline\pi:\overline G\to A$. Let $D$ be the sum of the zero and infinity divisors in the projective-line factors of $\overline G$, and put $M=\mathcal O_{\overline G}(D)$. Choose a symmetric ample line bundle $N$ on $A$, and set
\[
    L=M+\overline\pi^*N.
\]
Here and below, tensor products of line bundles are written additively. If $t=0$, take $\overline G=A$ and $M=0$. The canonical height associated with these data is
\[
    \hh_L(x)=\hh_M(x)+\hh_N(\pi(x)),\qquad x\in G(\overline K),
\]
where $\hh_M$ is the canonical height associated with the boundary line bundle $M$ and $\hh_N$ is the N\'eron--Tate height associated with $N$ on $A$. Both summands are nonnegative. Under multiplication by a positive integer, $\hh_M$ is homogeneous of degree one and $\hh_N$ is homogeneous of degree two.

We fix these data throughout the paper. For $\epsilon>0$, put
\[
    B_\epsilon(G)=\{x\in G(\overline K):\hh_L(x)<\epsilon\}.
\]
The set $B_\epsilon(G)$ depends on $L$. Heights from two standard choices are comparable by positive multiplicative constants; consequently, the condition below involving every $\epsilon>0$ is independent of the choice. The comparison is recalled in Section~\ref{sec:heights}. We say that a subvariety contains a dense set of small points if its intersection with $B_\epsilon(G)$ is dense for every $\epsilon>0$.

Let $\Gamma_0\subset G(\overline K)$ be a finitely generated subgroup, and put
\[
    \Gamma=\Gamma_0^{\mathrm{div}}:=\{x\in G(\overline K):[n]x\in\Gamma_0\text{ for some }n\geq1\},
    \qquad
    \Gamma_\epsilon=\Gamma+B_\epsilon(G).
\]
Following Poonen \cite[\S1]{Poo99}, we call $\Gamma$ the division group of $\Gamma_0$ and may visualize $\Gamma_\epsilon$ as a fattening of $\Gamma$, a "slab" in the topology on $G(\overline{K})$. For every $n\geq1$, it satisfies $[n]^{-1}\Gamma=\Gamma$, where the inverse image is taken in $G(\overline K)$. In particular, $\Gamma$ contains all torsion points. A subgroup $\Lambda$ has finite rank if $\dim_{\bQ}(\Lambda\otimes_{\bZ}\bQ)<\infty$. The groups $\Gamma_0^{\mathrm{div}}$ are exactly the finite-rank subgroups $\Lambda$ satisfying $[n]^{-1}\Lambda=\Lambda$ for every $n\geq1$.

\subsection{The main theorem}

Poonen \cite[Theorem~8 and Corollary~9]{Poo99} proved Mordell--Lang plus Bogomolov over number fields for division groups on almost split semi-abelian varieties, that is, varieties isogenous to a product of an abelian variety and a torus. Zhang \cite[Theorem~1.2]{Zha00} independently proved the abelian case by equidistribution of almost division points. R\'emond \cite[Theorems~1.1 and~1.2]{Rem03} treated arbitrary semi-abelian varieties over $\overline\bQ$ and finite-rank subgroups. 

We formulate and prove the geometric analogue. Since geometric canonical heights vanish on constant points and their images under homomorphisms, constant subvarieties must occur in the conclusion even when they are not cosets.

Let $X\subset G_{\overline K}$ be a closed integral subvariety. Let $\Stab_{G_{\overline K}}(X)$ be its stabilizer, and set
\[
    S=\Stab_{G_{\overline K}}(X)^0,\qquad
    q:G_{\overline K}\longrightarrow\widetilde G=G_{\overline K}/S.
\]
Here $S$ is the connected stabilizer of $X$; quotienting by $S$ makes the stabilizer of $q(X)$ finite. We say that $X$ is \emph{$\Gamma$-special} if there exist $\gamma\in\Gamma$, a semi-abelian variety $G_0$ over $k$, a homomorphism $h:(G_0)_{\overline K}\to\widetilde G$ with finite kernel, and a closed integral $k$-subvariety $Y\subset G_0$ such that
\begin{equation}\label{eq:special}
    q(X)=q(\gamma)+h(Y_{\overline{K}}).
\end{equation}
This is the notion used in Luo--Yu's geometric Bogomolov theorem \cite[Theorem~1.1]{LY2025Bogomolov}, with the torsion translation replaced by an element of $\Gamma$.

There is an equivalent formulation using the semi-abelian Chow trace. Let
\[
    (T_X,\tau_X)=\Tr_{\overline K/k}(\widetilde G).
\]
The group $T_X$ is defined over $k$, and $\tau_X:(T_X)_{\overline K}\to\widetilde G$ is universal for homomorphisms from constant semi-abelian varieties; see \cite[Corollary~2.4.4(2)]{Liu2024chowtrace1motiveslangneron} and \cite{Liu2026pushforward}. Then $X$ is $\Gamma$-special if and only if there exist $\gamma\in\Gamma$ and a closed integral $k$-subvariety $Y'\subset T_X$ such that
\begin{equation}\label{eq:special-trace}
    q(X)=q(\gamma)+\tau_X(Y'_{\overline{K}}).
\end{equation}
Section~\ref{sec:heights} proves the equivalence. In both formulations we first pass to the stabilizer quotient.

\begin{theorem}\label{thm:main}
    With the notation above, the following conditions are equivalent.
    \begin{enumerate}[beginpenalty=10000]
        \item For every $\epsilon>0$, the set
        \[
            X(\overline{K})\cap\Gamma_\epsilon
        \]
        is Zariski dense in $X$.
        \item $X$ is $\Gamma$-special, equivalently it satisfies \eqref{eq:special}, or equivalently \eqref{eq:special-trace}.
    \end{enumerate}
\end{theorem}

When $\Gamma_0=\{0\}$, the group $\Gamma$ is the torsion subgroup, and Theorem~\ref{thm:main} recovers Luo--Yu's geometric Bogomolov theorem: translation by torsion preserves canonical heights. For a group of positive rank, the approximating elements of $\Gamma$ can have unbounded height. Theorem~\ref{thm:vojta} supplies the estimate needed to handle them. Even after this estimate, bounded height does not imply finiteness, since every constant point has height zero; Proposition~\ref{prop:bounded} instead classifies the subvarieties on which points of bounded height approach $\Gamma$ arbitrarily closely.

Induction on dimension gives a single height threshold and a finite family of special subvarieties, as in \cite[Introduction, Conjectures~1 and~2]{Poo99}.

\begin{corollary}\label{cor:exceptional}
    With the notation above, there are $\epsilon_0>0$ and a finite, possibly empty, family of integral $\Gamma$-special subvarieties $Z_1,\ldots,Z_m\subset X$ such that
    \[
        X(\overline{K})\cap\Gamma_{\epsilon_0}\subseteq\bigcup_{i=1}^m Z_i(\overline{K}).
    \]
    If $X$ is not $\Gamma$-special, all the $Z_i$ may be taken to be proper subvarieties of $X$.
\end{corollary}

The method in this paper does not work in positive characteristic; see \S\ref{sec:charp}.

\subsection{Related results}
Mordell conjectured that a curve over $\bQ$ of genus at least two has finitely many rational points. Faltings proved the number-field form \cite[Satz~7]{Fal83}: for a smooth projective geometrically integral curve $C$ of genus at least two over a number field $F$, the set $C(F)$ is finite. A rational base point, when one exists, embeds $C$ in its Jacobian. The Mordell--Weil theorem makes the Jacobian's group of $F$-rational points finitely generated, placing Mordell's conjecture in a more general intersection problem.

Lang \cite[p.~29]{Lan60} formulated this problem for tori and abelian varieties and asked for its extension to semi-abelian varieties. Faltings \cite{Fal91,Fal94} proved the abelian case, Vojta \cite[Theorem~0.2 and pp.~133--134]{Voj96} the semi-abelian case for finitely generated groups, and McQuillan \cite[Introduction and \S3]{McQ95} the finite-rank case. Thus the intersection of a subvariety of a complex semi-abelian variety with a finite-rank subgroup is a finite union of intersections of that subgroup with semi-abelian cosets contained in the subvariety; each coset can be translated by a point of the subgroup. This gives Mordell's conjecture because a curve of genus at least two in its Jacobian contains no positive-dimensional abelian translate. The torsion case is the Manin--Mumford theorem. Han--Luo--Yu \cite[Theorem~1.1]{HLY26} give a bound for the number of cosets in terms of dimension, degree, and rank in characteristic zero. These are statements about exact intersections.

Over function fields of characteristic zero, Manin \cite{Man63} and Grauert \cite[p.~132]{Gra65} proved the geometric Mordell theorem: a smooth projective geometrically integral non-isotrivial curve of genus at least two has finitely many $K$-rational points. Here a curve is non-isotrivial if it does not become defined over $k$ after any finite extension of $K$. Constant curves have Zariski dense constant points, which explains the non-isotriviality assumption. Hrushovski \cite[Theorem~1.1]{Hru96} proved the relative function-field Mordell--Lang theorem in all characteristics; see \cite[\S4]{BBP18} for its semi-abelian formulation. In characteristic zero, the subgroup may have finite rational rank. In characteristic $p>0$, the cited formulation requires containment in the prime-to-$p$ divisible hull of a finitely generated group: every point must have a multiple in that group with multiplier prime to $p$. Finite rational rank alone does not imply this condition.

Over number fields, Ullmo \cite{Ull98} proved Bogomolov's conjecture for curves in their Jacobians, and Zhang \cite{Zha98} proved it for subvarieties of abelian varieties. The conclusion is that a subvariety with dense small points is a torsion translate of an abelian subvariety.

For geometric heights, constant subvarieties also occur. Gubler \cite[Theorem~1.1]{Gub07} proved the theorem when the abelian variety is totally degenerate at one place, using tropical equidistribution. Yamaki \cite[Theorem~1.5]{Yam18} reduced the general abelian conjecture to nowhere-degenerate abelian varieties with trivial trace; here nowhere degenerate means potentially good reduction at every divisorial place. Gao--Habegger \cite[Theorem~1.1 and Appendix~A]{GH19} proved the abelian geometric Bogomolov theorem in characteristic zero and transcendence degree one. Cantat--Gao--Habegger--Xie \cite[Theorem~A]{CGHX21} treated arbitrary transcendence degree in characteristic zero, and Xie--Yuan \cite[Theorem~1.1]{XY22} proved the theorem in arbitrary characteristic. Luo--Yu \cite[Theorem~1.1]{LY2025Bogomolov} extended it to semi-abelian varieties. Their Example~1.2 explains why one must allow positive heights tending to zero: a special subvariety of a nonconstant extension need not have dense height-zero points.

The closest arithmetic estimate to our argument is R\'emond's Vojta inequality \cite[Theorem~4.1]{Rem03}. His approximation theorem \cite[Theorem~1.2]{Rem03} permits errors $z$ satisfying
\[
    \hh_L(z)\leq\epsilon\bigl(1+\hh_L(\gamma)\bigr)
\]
for the corresponding arithmetic canonical height, and gives finiteness outside positive-dimensional cosets. In the abelian case, Ge \cite[Theorems~1.2 and~1.2$'$]{Ge24} chooses the height threshold in terms of the dimension of the ambient variety and the degree of the subvariety, and bounds the number of cosets needed to cover the intersection exponentially in the rank of the group.

Moriwaki \cite[Theorem~A]{Mor01} treats abelian varieties over finitely generated extensions of $\bQ$ using big arithmetic polarizations; those heights include arithmetic contributions. Hultberg \cite[Theorems~2 and~4]{Hul26} proves Bogomolov for semi-abelian varieties over globally valued fields of characteristic zero and a gap principle over non-archimedean globally valued fields for small points on subvarieties with finite stabilizer whose differences generate the ambient group. The Vojta inequality below concerns geometric heights over the fixed algebraically closed constant field. For fixed geometric data, its constants are independent of the finite extension of $K$ over which the points are defined, and the auxiliary comparison is uniform in the integer weights.

After completing the proof of Theorem~\ref{thm:main} in August 2026, the author learned in September 2026 of an independent proof of a similar result by Ningjun Jiang. The proof presented here was developed independently.

\subsection{Plan of proof}

The implication from specialness to the density of approximations follows from geometric Bogomolov. For the converse, pass to the connected stabilizer quotient; the image of $\Gamma$ remains a finite-rank division group. We are then reduced to a subvariety $X$ with finite stabilizer.

First suppose that the approximating points on $X$ have bounded height. Sections~\ref{sec:heights} and~\ref{sec:bounded} use the seminorm
\[
    \rho(x)=\hh_M(x)+\sqrt{\hh_N(\pi(x))},
\]
whose two summands have the same degree of homogeneity. Finite rank gives a finite covering of the relevant bounded part of $\Gamma$ by sets of small $\rho$-diameter. One part contains approximations to a dense subset of $X$. Taking differences cancels its center. Thus, for each integer $m\geq2$, the image of
\[
    X^m\longrightarrow G^{m-1},\qquad (x_1,\ldots,x_m)\longmapsto(x_2-x_1,\ldots,x_m-x_1)
\]
has dense small points. For sufficiently large $m$, the image has finite stabilizer. Geometric Bogomolov makes it constant up to torsion translation, and a fiber of the next difference map gives $X=a+Y_{\overline K}$ with $Y$ defined over $k$.

Determining the translate requires more than geometric Bogomolov. A generically finite difference map on $Y$, defined over $k$, gives a height inequality without an additive error. It follows that
\[
    \inf_{\gamma\in\Gamma}\rho(a-\gamma)=0.
\]
The positive gap for nonzero heights over each fixed finite extension of $K$ implies that the real kernel of $\rho$ on a finite-dimensional rational subspace is defined over $\bQ$. The identities $[n]^{-1}\Gamma=\Gamma$ then give $a=\gamma+z$ with $\gamma\in\Gamma$ and $\hh_L(z)=0$, as required.

To remove the height bound, let $Z_X$ be the union of the positive-dimensional semi-abelian cosets contained in $X$. The Vojta inequality, Theorem~\ref{thm:vojta}, excludes tuples in $X\setminus Z_X$ whose heights grow rapidly and whose normalized differences are small. The two height summands require different normalizations. Covering the corresponding normalized images of $\Gamma$ by finitely many small sets gives the height bound of Proposition~\ref{prop:large}. Since $X$ has finite stabilizer, $Z_X$ is a proper closed subset; removing it preserves density and reduces the theorem to the bounded-height case.

The proof of the Vojta inequality occupies Section~\ref{sec:vojta}. We first work over a curve. The geometric product estimate retains the Samuel multiplicities in the derivative intersections and bounds the degree and height of each product factor. The comparison obtained from the auxiliary section then descends to products of smaller dimension with leading constants independent of the integer weights and the finite extensions of $K$ over which the points are defined. Fixing those weights and a finite extension of $K$ containing the points before letting the auxiliary degree tend to infinity removes the genus term in Siegel's lemma. Finally, one generic complete-intersection curve preserves normalized heights, up to a fixed factor, over every finite extension of $K$ and gives the result for the original polarization.

\subsection{Notation and terminology}

A variety is an integral separated scheme of finite type over a field, and a subvariety is closed and integral. A subscript denotes extension of scalars. A constant variety over $\overline K$ is the base change of a variety over $k$. Products and stabilizers are taken over $\overline K$ unless another field is specified, and density means Zariski density. We use multiplicative notation only in the torus examples.

\subsection{Acknowledgement and AI usage}
The author is grateful to Professor Huayi Chen for his encouragement and thoughtful advice during the preparation of this manuscript.

The ideas are due to the author's brain. The author used ChatGPT to polish the paper and verify the mathematical correctness. The author thanks OpenAI's ChatGPT for assistance with the presentation of the manuscript.

\section{Heights and approximation of a point}\label{sec:heights}

We prove that a point approximated arbitrarily closely by $\Gamma_0^{\mathrm{div}}$, with $\Gamma_0$ finitely generated, belongs to $\Gamma_0^{\mathrm{div}}$ modulo a point of height zero. The proof uses a seminorm associated with the height and a positive lower bound for nonzero heights over each fixed finite extension of $K$.

\subsection{Canonical heights}

Keep the line bundles $M$, $N$, and $L$ chosen in the introduction. If $h_M$ and $h_N$ are corresponding Weil heights, Tate's limiting process gives, for $x\in G(\overline K)$ and $a\in A(\overline K)$,
\[
    \hh_M(x)=\lim_{n\to\infty}\frac{h_M([n]x)}{n},\qquad
    \hh_N(a)=\lim_{n\to\infty}\frac{h_N([n]a)}{n^2}.
\]
Let $\overline M$ and $\overline N$ be the resulting canonically metrized adelic line bundles. Following Luo--Yu \cite[\S2.4--\S2.5]{LY2025Bogomolov}, put $\overline L=\overline M+\overline\pi^*\overline N$, so that $\hh_L=h_{\overline L}$. These heights are the geometric counterparts of those in \cite[\S3 and Lemma~4]{Poo99}. For every integer $n$, they satisfy
\begin{equation}\label{eq:scaling}
    \hh_M([n]x)=|n|\hh_M(x),\qquad
    \hh_N([n]a)=n^2\hh_N(a)\qquad(n\in\bZ).
\end{equation}
To give the two summands the same scaling under multiplication, put
\begin{equation}\label{eq:rho}
    \rho(x)=\hh_M(x)+\sqrt{\hh_N(\pi(x))}.
\end{equation}
Both summands are homogeneous of degree one and satisfy the triangle inequality. For $\hh_M$, apply the geometric Weil-height inequality obtained from the effective-divisor inequality in \cite[Corollary~3.1]{Rem03} to $[n]x$ and $[n]y$, divide by $n$, and let $n$ tend to infinity. For $\sqrt{\hh_N}$, use the nonnegativity of the quadratic form $\hh_N$. Hence
\[
    \rho(x+y)\leq\rho(x)+\rho(y),\qquad \rho([n]x)=|n|\rho(x).
\]
Each of $\hh_M$, $\sqrt{\hh_N\circ\pi}$, and $\rho$ vanishes on torsion and induces a seminorm on $G(\overline{K})\otimes_\bZ\bQ$. These seminorms extend continuously to finite-dimensional real scalar extensions: if $e_1,\ldots,e_r$ is a rational basis, the bound $\rho(\sum_i a_i e_i)\leq\sum_i|a_i|\rho(e_i)$ gives a unique continuous extension, and the same argument applies to the other two seminorms.

The elementary bounds
\begin{equation}\label{eq:compare}
    \rho(x)\leq\hh_L(x)+\sqrt{\hh_L(x)},\qquad
    \hh_L(x)\leq\rho(x)+\rho(x)^2
\end{equation}
allow us to replace $\hh_L$ by $\rho$ when considering boundedness or convergence to zero. For a product with chosen line bundles $L_1,\ldots,L_r$ defining the heights, we use the external tensor product, whose height is
\[
    \hh_{L_1\boxtimes\cdots\boxtimes L_r}(x_1,\ldots,x_r)=\sum_{i=1}^r\hh_{L_i}(x_i).
\]

Let $G'$ be another semi-abelian variety, with abelian quotient $A'$. Choose its standard compactification, boundary line bundle $M'$, and symmetric ample line bundle $N'$ on $A'$, and put $L'=M'+\overline\pi'^*N'$, where $\overline\pi'$ extends $G'\to A'$. For a homomorphism $f:G\to G'$, \cite[Lemma~3.3, proof of (1)]{LY2025Bogomolov} gives $\hh_{L'}(f(x))\leq C\hh_L(x)$, where $C>0$ depends on $f$ and the height data. If $f$ has finite kernel, the reverse comparison also holds. This is part~(2) of the same lemma for isogenies. In general, choose character bases so that the map on the tori has the form
\[
    (t_1,\ldots,t_r)\longmapsto(t_1^{d_1},\ldots,t_r^{d_r},1,\ldots,1),\qquad d_i\geq1.
\]
The target boundary pulls back to $\sum_i d_iD_i$, where $D_i$ is the sum of the zero and infinity divisors in the $i$th source factor. Nonnegativity of their canonical heights gives the reverse comparison for $\hh_M$. The map on abelian quotients is finite onto its image. Indeed, the inverse image of a positive-dimensional connected kernel would map with finite kernel to the target torus; it would therefore be affine, contradicting its positive-dimensional abelian quotient. The pullback of $N'$ is consequently ample and gives the reverse comparison for $\hh_N\circ\pi$. Applying the upper bound to an isomorphism and its inverse shows that changing the torus coordinates preserves these comparisons.

We next identify the points of height zero. Let $(G^{\overline{K}/k},\tr)$ be the geometric $\overline{K}/k$-trace of $G_{\overline{K}}$; see \cite[Corollary~2.4.4(2)]{Liu2024chowtrace1motiveslangneron} and \cite[\S3.2]{LY2025Bogomolov}. Thus every homomorphism from a constant semi-abelian variety to $G_{\overline{K}}$ factors uniquely through $\tr$. Its connected reduced kernel descends to $k$ by \cite[Theorem~2.3.12]{Liu2024chowtrace1motiveslangneron}. The quotient by this kernel is constant and maps to $G_{\overline K}$, so universality gives a map from the quotient back to $G^{\overline{K}/k}$. Uniqueness in the universal property makes the composite with the quotient map the identity. The connected reduced kernel is therefore zero, and $\tr$ has finite kernel. In characteristic zero, every finite subgroup of a constant semi-abelian variety is defined over $k$. Its quotient, and hence the image of $\tr$, is constant. Write $H$ for this image with its induced $k$-structure.

The universal property also identifies the two descriptions of specialness in the introduction. For the quotient $\widetilde G$, it factors the homomorphism in \eqref{eq:special} as $h=\tau_X\circ v_{\overline{K}}$, where $v:G_0\to T_X$ is defined over $k$. The kernel of $v$ is finite because that of $h$ is finite. Thus $v$ is finite onto its image, and $Y'=v(Y)$ is an integral closed $k$-subvariety of $T_X$. This gives \eqref{eq:special-trace}. Conversely, \eqref{eq:special-trace} gives \eqref{eq:special} with $G_0=T_X$ and $h=\tau_X$.

By \cite[Theorem~1.1]{LY2025Bogomolov} applied to a point,
\begin{equation}\label{eq:zero}
    \{x\in G(\overline{K}):\hh_L(x)=0\}=G(\overline{K})_{\tors}+H(k).
\end{equation}

\subsection{A lower bound for positive heights}

Lang--N\'eron and positive definiteness modulo the trace bound the positive values of $\hh_N$ from below over each fixed finite extension of $K$; see \cite[Theorems~7.1 and~9.15]{Con06}. Above points of quadratic height zero, we express $\hh_M$ as an integral boundary degree divided by a fixed denominator. This proves the following analogue of \cite[Lemma~10(5)]{Poo99}.

\begin{lemma}\label{lem:gap}
    For every finite extension $F/K$, there is $a_F>0$ such that
    \[
        P\in G(F),\quad\hh_L(P)<a_F\quad\Longrightarrow\quad\hh_L(P)=0.
    \]
\end{lemma}

\begin{proof}
    A lower bound after a finite extension of $F$ also applies to $G(F)$. We may thus enlarge $F$ so that the splitting, compactification, and line bundles defining the heights are defined over it. Let $(A^{F/k},\tr_{F/k})$ be the $F/k$-trace of $A_F$. The group
    \[
        A(F)/\tr_{F/k}(A^{F/k}(k))
    \]
    is finitely generated by Lang--N\'eron, and its N\'eron--Tate height is positive definite after real scalar extension; see \cite[Theorems~7.1 and~9.15]{Con06}. Modulo torsion, this group is a lattice, so its positive height values have a positive lower bound. Fix an integer $e\geq1$ killing the torsion in the displayed quotient. If $\hh_L(P)$ is below the positive lower bound, then $\hh_N(\pi(P))=0$ and
    \[
        [e]\pi(P)=\tr_{F/k}(a)\qquad\text{for some }a\in A^{F/k}(k).
    \]
    The integer $e$ is independent of $P$. We now bound the positive linear heights above the trace image by $[F:K]^{-1}$.

    Put $A_0=A^{F/k}$ and pull $G_F\to A_F$ back to $(A_0)_F$. The resulting extension by the split torus is determined by rigidified algebraically trivial line bundles, hence by finitely many points $\eta_j\in A_0^\vee(F)$. Let $\mathcal B_F$ be the normalization of $\mathcal B$ in $F$, and let $\mathcal H_F$ be the pullback of $\mathcal H$. Properness of $A_0^\vee$ makes every rational map $\mathcal B_F\dashrightarrow A_0^\vee$ defined by an $\eta_j$ regular at the codimension-one points. All these maps are therefore morphisms on one open subset $U\subset\mathcal B_F$ whose complement has codimension at least two.

    The rigidified Poincar\'e bundles give the pulled-back extension and its standard compactification over $A_0\times U$. Let $D$ be the sum of the zero and infinity divisors in its projective-line factors. The multiplication morphism satisfies
    \begin{equation}\label{eq:modelboundary}
        [n]^*D=nD\qquad(n\geq1).
    \end{equation}
    This follows from the rigidified identity $[n]^*Q\simeq Q^{\otimes n}$ for each algebraically trivial extension bundle $Q$: in the associated projective-line bundle, the fiber coordinates are raised to their $n$th powers.

    A point $P'$ of the pulled-back extension above $a\in A_0(k)$ gives a rational section over $U$. It extends at codimension-one points by properness of the compactification. Its contact with $D$ is an effective integral Weil divisor $D_{P'}$ on $\mathcal B_F$. Its model height is
    \begin{equation}\label{eq:boundarydegree}
        \frac{\deg_{\mathcal H_F}(D_{P'})}{[F:K]},
        \qquad
        \deg_{\mathcal H_F}(D_{P'})=\mathcal H_F^{\dim\mathcal B-1}\cdot D_{P'}.
    \end{equation}
    The generic-fiber boundary is the pullback of the chosen boundary on $G$, so its model height differs from the corresponding Weil height by a bounded function. By \eqref{eq:modelboundary}, the model height scales exactly by $n$ under $[n]$. Dividing by $n$ and taking the limit therefore identifies \eqref{eq:boundarydegree} with $\hh_M$ at the image of $P'$ in $G(F)$. A nonzero effective integral divisor has positive integral $\mathcal H_F$-degree. Hence this height is either zero or at least $[F:K]^{-1}$.

    Apply this to the lift of $[e]P$ above $a$. Since $\hh_M([e]P)=e\hh_M(P)$, a point with $\hh_N(\pi(P))=0$ has either zero linear height or linear height at least $(e[F:K])^{-1}$. Choose $a_F$ smaller than this bound and the positive gap for $\hh_N\circ\pi$. Then $\hh_L(P)<a_F$ forces both summands to vanish.
\end{proof}

The height gap implies that the real kernel of $\rho$ is rational on every finitely generated subgroup. The lattice argument is the seminorm analogue of \cite[Lemma~10.1]{Con06}.

\begin{lemma}\label{lem:rational}
    Let $\Lambda\subset G(\overline{K})$ be a finitely generated subgroup. The kernel of $\rho$ on $\Lambda\otimes_\bZ\bR$ is defined over $\bQ$ with respect to $\Lambda\otimes_\bZ\bQ$.
\end{lemma}

\begin{proof}
    Choose a finite extension $F/K$ over which generators of $\Lambda$ are defined. Lemma~\ref{lem:gap} and \eqref{eq:compare} give a positive lower bound for the nonzero $\rho$-values on $\Lambda/\Lambda_{\tors}$. Its image in $(\Lambda\otimes_\bZ\bR)/\ker\rho$ is therefore discrete. Since the image spans the quotient, it is a lattice of rank equal to the dimension of the quotient. The kernel of the map from $\Lambda/\Lambda_{\tors}$ to this lattice consequently has rank $\dim_\bR\ker\rho$ and spans $\ker\rho$.
\end{proof}

Applying this rationality statement to a point and a finite set of generators gives the following analogue of \cite[Lemma~14]{Poo99}. Here the term of height zero includes constant points of the trace, as in \eqref{eq:zero}.

\begin{corollary}\label{cor:point}
    Let $\Gamma\subset G(\overline{K})$ be a subgroup of finite rank satisfying $[n]^{-1}\Gamma=\Gamma$ for every $n\geq1$. If $a\in G(\overline{K})$ satisfies
    \[
        \inf_{\gamma\in\Gamma}\rho(a-\gamma)=0,
    \]
    then $a=\gamma+z$ for some $\gamma\in\Gamma$ and some point $z$ with $\hh_L(z)=0$.
\end{corollary}

\begin{proof}
    Choose $g_1,\ldots,g_r\in\Gamma$ whose images span $\Gamma\otimes_\bZ\bQ$, and apply Lemma~\ref{lem:rational} to the group generated by $a,g_1,\ldots,g_r$. Quotient its real scalar extension by the kernel of $\rho$. The approximation hypothesis places the image of $a$ in the closure of the span of the images of the $g_i$. This span is finite-dimensional, hence closed. Both it and the kernel of $\rho$ are rational. Solving the corresponding linear equations over $\bQ$ gives rational numbers $c_i$ such that
    \[
        \rho\Bigl(a-\sum_i c_i g_i\Bigr)=0
    \]
    in the rational scalar extension. Choose an integer $n\geq1$ with $nc_i\in\bZ$. Multiplication by $n$ is surjective on $G(\overline{K})$, so there is $\gamma\in G(\overline{K})$ with $[n]\gamma=\sum_i nc_i g_i$. Since $[n]\gamma\in\Gamma$ and $[n]^{-1}\Gamma=\Gamma$, we have $\gamma\in\Gamma$. Homogeneity of $\rho$ now yields $\rho(a-\gamma)=0$, and \eqref{eq:compare} finishes the proof.
\end{proof}

\section{Approximation with bounded heights}\label{sec:bounded}

We characterize approximation by a finite-rank subgroup when the approximating points have bounded height. Over $\overline\bQ$, R\'emond \cite[\S5, pp.~209--211]{Rem03} combines a finite covering argument with Bogomolov's theorem. Over a function field, the conclusion involves a constant subvariety of the trace. Retain $H$ and $\rho$ from Section~\ref{sec:heights}, and put $\Lambda_\epsilon=\Lambda+B_\epsilon(G)$ for a subgroup $\Lambda\subset G(\overline{K})$.

\begin{proposition}\label{prop:bounded}
    Let $\Lambda\subset G(\overline{K})$ be a subgroup of finite rank, and let $X\subset G_{\overline{K}}$ have finite stabilizer. The following are equivalent.
    \begin{enumerate}[beginpenalty=10000]

        \item There is $R\geq0$ such that, for every $\epsilon>0$, the set
        \[
            \{x\in X(\overline{K})\cap\Lambda_\epsilon:\hh_L(x)\leq R\}
        \]
        is dense in $X$.

        \item There are $a\in G(\overline{K})$ and a closed $k$-subvariety $Y\subset H$ such that
        \[
            X=a+Y_{\overline{K}},\qquad \inf_{\lambda\in\Lambda}\rho(a-\lambda)=0.
        \]
    \end{enumerate}
    If $[n]^{-1}\Lambda=\Lambda$ for every $n\geq1$, these conditions imply that $X$ is $\Lambda$-special.
\end{proposition}

\subsection{Difference maps}

Under condition~\textup{(1)}, we first show that a translate of $X$ is defined over $k$. The finite covering argument of \cite[Lemma~5.1(3)]{Rem03} produces a dense subset close to one element of $\Lambda$. Taking differences cancels that element and gives small points on the image of the difference morphism.

\begin{lemma}\label{lem:clusters}
    Assume condition~\textup{(1)} of Proposition~\ref{prop:bounded}. There are points $\lambda_n\in\Lambda$, positive numbers $b_n\to0$, and dense subsets $E_n\subset X(\overline{K})$ such that
    \[
        \rho(x-\lambda_n)\leq b_n\qquad(x\in E_n).
    \]
\end{lemma}

\begin{proof}
    Choose $\epsilon_n>0$ with $\epsilon_n+\sqrt{\epsilon_n}\leq1/n$. If $x\in X(\overline{K})\cap\Lambda_{\epsilon_n}$ and $\hh_L(x)\leq R$, choose $\lambda\in\Lambda$ with $\hh_L(x-\lambda)<\epsilon_n$. Then
    \[
        \rho(x-\lambda)\leq 1/n,\qquad \rho(\lambda)\leq R+\sqrt R+1.
    \]
    The points $\lambda$ satisfying the second inequality have bounded image in the finite-dimensional normed space $(\Lambda\otimes_\bZ\bR)/\ker\rho$. Cover this bounded image by finitely many balls of radius $1/n$ with centers in the image, and lift the centers to $\Lambda$. The dense set under consideration is then covered by finitely many sets of the form $\rho(x-\lambda)\leq2/n$. Irreducibility of $X$ forces one of these sets to be dense. Take it for $E_n$, its center for $\lambda_n$, and put $b_n=2/n$.
\end{proof}

For an integer $m\geq2$, define the difference morphism
\[
    \alpha_m:X^m\longrightarrow G^{m-1},\qquad
    (x_1,\ldots,x_m)\longmapsto(x_2-x_1,\ldots,x_m-x_1),
\]
and let $W_m$ be the closure of its image. Zhang \cite[Lemma~3.1]{Zha98} proves generic finiteness for sufficiently many factors in an abelian variety when the stabilizer is trivial. The following argument applies to semi-abelian varieties with finite stabilizer and also controls the stabilizer of $W_m$.

\begin{lemma}\label{lem:incidence}
    Let $m\geq2$ satisfy $m>\dim G$, and suppose that $X$ has finite stabilizer. Then
    \begin{enumerate}[beginpenalty=10000]

        \item $\alpha_m$ is generically finite onto $W_m$;

        \item $W_m$ has finite stabilizer.
    \end{enumerate}
\end{lemma}

\begin{proof}
    If $X$ is a point, both assertions are immediate. Assume $d=\dim X>0$, and put $g=\dim G$ and $F=\Stab(X)(\overline{K})$. For $t\notin F$, the distinct irreducible varieties $X$ and $X-t$ satisfy
    \[
        \dim\bigl(X\cap(X-t)\bigr)\leq d-1.
    \]
    The incidence of tuples $(x_1,\ldots,x_m,t)$ with $t\notin F$ and $x_i+t\in X$ for every $i$ has dimension at most
    \[
        g+m(d-1)<md.
    \]
    It cannot dominate $X^m$. Two tuples in a fiber of $\alpha_m$ differ by one common translation. A general tuple thus admits only translations in the finite group $F$, proving~(1).

    Let $v=(v_2,\ldots,v_m)$ stabilize $W_m$, and put $v_1=0$. The constructible image of $\alpha_m$ contains a dense open subset of $W_m$. Therefore, for a general tuple $(x_i)\in X^m$, the point $\alpha_m((x_i))+v$ has a preimage. Equivalently, there is $t\in G(\overline{K})$ such that
    \[
        x_i+t+v_i\in X\qquad(1\leq i\leq m).
    \]
    Suppose that $v\notin F^{m-1}$. Outside the finite set $\bigcup_i(F-v_i)$, every intersection $X\cap(X-t-v_i)$ has dimension at most $d-1$, so the corresponding incidence again has dimension at most $g+m(d-1)$. For a point $t$ in that finite set, at least one $t+v_i$ lies outside $F$: otherwise $t\in F$ because $v_1=0$, and then all $v_i\in F$. The incidence over each such $t$ consequently has dimension at most $md-1$. Neither part dominates $X^m$, a contradiction. Thus $\Stab(W_m)(\overline{K})\subseteq F^{m-1}$, proving~(2).
\end{proof}

For sufficiently large $m$, finite stabilizer and geometric Bogomolov \cite[Theorem~1.1]{LY2025Bogomolov} make $W_m$ constant when it contains a dense set of small points. A fiber of $W_{m+1}\to W_m$ over a constant point then recovers a translate of $X$.

\begin{lemma}\label{lem:descent}
    Suppose that $X$ has finite stabilizer and every $W_m$ contains a dense set of small points. Then $X=a+Y_{\overline{K}}$ for some $a\in G(\overline{K})$ and some closed $k$-subvariety $Y\subset H$.
\end{lemma}

\begin{proof}
    Choose $m>\dim G$ with $m\geq2$. Lemma~\ref{lem:incidence} gives a finite stabilizer for $W_m$. Since the trace of a product is the product of the traces, \cite[Theorem~1.1]{LY2025Bogomolov} makes $W_m$ a torsion translate of a constant subvariety of $H^{m-1}$. Diagonal tuples show that $0\in W_m$, so the translating torsion point lies in $H^{m-1}$. It is constant in characteristic zero. Hence $W_m$ is defined over $k$. The same holds for $W_{m+1}$.

    The coordinate projections of $W_{m+1}$ contain all differences $x-x'$ with $x,x'\in X(\overline{K})$. Hence $X-X\subset H$. Consider the projection
    \[
        p:W_{m+1}\longrightarrow W_m
    \]
    which forgets the last difference. Put $d=\dim X$. By Lemma~\ref{lem:incidence}, its source and target have dimensions $(m+1)d$ and $md$. There is a nonempty open subset of $W_m$ contained in the image of $\alpha_m$, over which the fibers of $\alpha_m$ are finite and the fibers of $p$ have dimension $d$.

    The set $W_m(k)$ remains dense after extension to $\overline K$, since $k$ is algebraically closed. Thus this open subset contains a point $w\in W_m(k)$ even though it need not be defined over $k$. Choose a preimage under $\alpha_m$ and denote its first coordinate by $a$.

    Varying the last point in $X^{m+1}$ shows that $X-a$ is contained in the fiber $p^{-1}(w)$, viewed in its last coordinate in $H$. It is closed, irreducible, and of dimension $d$, so it is an irreducible component of that fiber. The fiber is defined over the algebraically closed field $k$, and each of its geometric irreducible components is defined over $k$. This proves the assertion.
\end{proof}

\subsection{The translating point}

To locate the translating point, we must deduce small heights on $Y$ from small heights of differences. The usual inequality for a generically finite map has an additive error; see \cite[Theorem~1]{Sil11}. Since the heights here tend to zero, we need to remove that error. The argument using a nonzero section in \cite[\S1]{Sil11} does so on models defined over $k$.

\begin{lemma}\label{lem:inverse}
    Let $H$ be a semi-abelian variety over $k$, and let $Y\subset H$ have finite stabilizer. Choose a standard compactification $\overline H$ of $H_{\overline{K}}$ and a line bundle $L_H=M_H+\overline\pi_H^*N_H$, where $\overline\pi_H$ extends the map from $H_{\overline{K}}$ to its abelian quotient, $M_H$ is the boundary line bundle, and $N_H$ is symmetric and ample on that quotient. There are an integer $r\geq2$, a nonempty open subset $U\subset Y^r$ defined over $k$, and $C>0$ such that
    \begin{equation}\label{eq:inverse}
        \sum_{i=1}^r\hh_{L_H}(y_i)\leq C\sum_{i=2}^r\hh_{L_H}(y_i-y_1)
        \qquad((y_1,\ldots,y_r)\in U(\overline{K})).
    \end{equation}
\end{lemma}

\begin{proof}
    Height comparison reduces the assertion to line bundles defined over $k$. Choose $r>\dim H$ with $r\geq2$. Compactify $Y^r$ and the image of its difference morphism over $k$, and let $Z$ be the closure of the graph of the difference map. Let $p$ and $f$ be the projections to these compactifications, respectively. Let $\mathcal L$ and $\mathcal N$ be the restrictions of the external products of the ample line bundles defining the heights on the corresponding powers of a standard compactification of $H$ defined over $k$. On the open parts, their geometric model heights are the sums in \eqref{eq:inverse}.

    By Lemma~\ref{lem:incidence}, $f$ is generically finite, so $f^*\mathcal N$ is big. The characterization of big line bundles in \cite[Corollary~2.2.7]{Laz04} gives positive integers $C,b$ and a nonzero section of
    \[
        (f^*\mathcal N)^{\otimes bC}\otimes(p^*\mathcal L)^{\otimes(-b)}
    \]
    over $k$. Remove the zero divisor of this section and the loci where the graph projections do not represent the original morphism. This gives a nonempty open subset $U\subset Y^r$ defined over $k$.

    Geometric model heights on constant models are exactly additive and functorial. They are nonnegative outside an effective constant divisor: pull its section back along the rational map from a normal model of a finite extension of $K$ defining the point. Properness extends this map at codimension-one points, where the resulting divisor is effective and has nonnegative polarized degree. Applying this to the section above gives
    \[
        bC\,h_{\mathcal N}(f(z))-b\,h_{\mathcal L}(p(z))\geq0
    \]
    for points above $U$. The constant model heights of the standard bundles equal their canonical heights: the multiplication identities already hold on those models. Dividing by $b$ yields \eqref{eq:inverse}.
\end{proof}

\begin{proof}[Proof of Proposition~\ref{prop:bounded}]
    Assume~(1), and use the sets $E_n$, centers $\lambda_n$, and numbers $b_n$ of Lemma~\ref{lem:clusters}. For every fixed $m$, the set $E_n^m$ is dense in $X^m$, and its differences satisfy
    \[
        \rho(x_i-x_1)\leq2b_n.
    \]
    Their image is dense in $W_m$. By \eqref{eq:compare}, the product height of each image point is at most $(m-1)(2b_n+4b_n^2)$. Consequently $W_m$ contains a dense set of small points. Lemma~\ref{lem:descent} gives $X=a+Y_{\overline{K}}$ with $Y\subset H$ defined over $k$.

    The stabilizer of $Y$ is finite. Choose the line bundle $L_H$ defining the height on $H$ and choose $r,U,C$ as in Lemma~\ref{lem:inverse}. For each $n$, density of $E_n^r$ gives a tuple $(x_1,\ldots,x_r)$ whose translate by $(-a,\ldots,-a)$ belongs to $U$. Write $x_i=a+y_i$ and $z_i=x_i-\lambda_n$; these points depend on $n$. Then
    \[
        y_i-y_1=z_i-z_1,\qquad \rho(z_i-z_1)\leq2b_n.
    \]
    By \eqref{eq:compare}, these differences have $\hh_L$-height at most $2b_n+4b_n^2$. Height comparison for $H\hookrightarrow G$ gives $\hh_{L_H}(y_i-y_1)\to0$. Since the constants in \eqref{eq:inverse} are fixed, we obtain $\hh_{L_H}(y_1)\to0$. Since
    \[
        a-\lambda_n=z_1-y_1,
    \]
    the triangle inequality, the bound $\rho(z_1)\leq b_n$, and height comparison give $\rho(a-\lambda_n)\to0$. This proves~(2).

    Conversely, assume~(2). The set $a+Y(k)$ is dense in $X$. Every $y\in Y(k)$ has height zero, so $\rho(a+y)=\rho(a)$, by applying the triangle inequality in both directions. Given $\epsilon>0$, choose $\lambda\in\Lambda$ with $\rho(a-\lambda)$ small enough that $\rho(a-\lambda)+\rho(a-\lambda)^2<\epsilon$. Then, for every $y\in Y(k)$,
    \[
        \hh_L(a+y-\lambda)\leq\rho(a-\lambda)+\rho(a-\lambda)^2<\epsilon.
    \]
    Also $\hh_L(a+y)\leq\rho(a)+\rho(a)^2$, so~(1) holds with $R=\rho(a)+\rho(a)^2$.

    Finally, if $[n]^{-1}\Lambda=\Lambda$ for every $n\geq1$, Corollary~\ref{cor:point} gives $a=\lambda+z$ with $\lambda\in\Lambda$ and $\hh_L(z)=0$. By \eqref{eq:zero}, write $z=\tau+c$ with $\tau$ torsion and $c\in H(k)$. Absorb $\tau$ into $\lambda$, since $\Lambda$ contains all torsion, and replace $Y$ by $c+Y$. This gives the required special form.
\end{proof}

\section{A Vojta inequality over function fields}\label{sec:vojta}

For a closed subvariety $V\subset G_{\overline{K}}$, let $Z_V$ be the union of the translates of positive-dimensional semi-abelian subvarieties contained in $V$. This is the \emph{Mordell exceptional locus} of \cite[\S0]{Abr94}. By \cite[Theorems~1 and~2]{Abr94}, it is closed, and $Z_V=V$ implies $\dim\Stab(V)>0$.

The following theorem is the function-field analogue of \cite[Theorem~4.1]{Rem03}. It excludes tuples outside $Z_V$ whose heights grow rapidly and whose normalized differences are small. The two normalizations reflect the linear growth of the boundary height and the quadratic growth of the abelian height under multiplication.

\begin{theorem}\label{thm:vojta}
    Suppose that $\dim V\geq1$, and put $m=\dim V+1$. There are positive constants $c_1,c_2,c_3$, depending on $G,V$, the line bundles $M,N$, and the fixed polarization, such that no tuple $(x_1,\ldots,x_m)\in(V\setminus Z_V)(\overline{K})^m$ satisfies all of
    \begin{enumerate}[beginpenalty=10000]

        \item $\hh_L(x_1)\geq c_3$ and $\hh_L(x_{i+1})\geq c_2^2\hh_L(x_i)$ for $1\leq i<m$;

        \item for $1\leq i<m$,
        \begin{align*}
            \hh_M\left(\frac{x_i}{\hh_L(x_i)}-\frac{x_{i+1}}{\hh_L(x_{i+1})}\right)&\leq c_1^{-1},\\
            \hh_N\left(\frac{\pi(x_i)}{\sqrt{\hh_L(x_i)}}-\frac{\pi(x_{i+1})}{\sqrt{\hh_L(x_{i+1})}}\right)&\leq c_1^{-2}.
        \end{align*}
    \end{enumerate}
    The expressions in \textup{(2)} are taken in the real scalar extension of the group generated by the points; the heights extend there by homogeneity and continuity. The constants are independent of the finite extensions of $K$ containing the points.
\end{theorem}

We first work over the function field of a curve. Proposition~\ref{prop:geometric-comparison}, applied to the semi-abelian difference maps, gives an inequality with integer weights. Approximating inverse square roots of the point heights by ratios of positive integers then gives the two normalized differences in the theorem; the boundary differences use the squares of these integers. The comparison must therefore have constants independent of the weights and the fields containing the points. We prove it by auxiliary sections, differentiation, and an index estimate, as in \cite[\S\S3--5]{Yua25}, with degree and height bounds for the smaller product factors produced by the index estimate. Sections~\ref{sec:mixed-completion} and~\ref{sec:vojta-completion} give the semi-abelian application and the passage to higher transcendence degree.

\subsection{Heights over a curve}\label{sec:aux-vector}

Suppose $K=k(C)$, where $C$ is a smooth projective curve over $k$. For a finite extension $F/K$, let $C_F$ be the smooth projective curve with function field $F$, equipped with its finite morphism to $C$. The absolute values $|a|_v=\exp(-\operatorname{ord}_v(a))$, indexed by the closed points $v$ of $C_F$, are trivial on $k^*$ and satisfy the product formula. For $z=(z_0:\cdots:z_N)\in\mathbb P^N(F)$, put
\[
    h(z)=\frac{1}{[F:K]}\sum_v\log\max_j|z_j|_v.
\]
This is the height $h_{\mathcal O(1)}$ for the constant model of the hyperplane line bundle. The product formula makes it independent of the homogeneous coordinates, and the normalization makes it invariant under finite extension of $F$.

For a nonzero polynomial, $h$ denotes the height of its coefficient vector. For a finite family of polynomials, it denotes the height of the joint coefficient vector, whose $v$-adic norm is the largest absolute value of any coefficient in the family. For a projective variety $Y$, let $h(Y)$ be its Chow height, divided by $[F:K]$ but not by $\deg Y$. Equivalently, this is the corresponding hyperplane intersection on the closure over $C_F$. These projective heights are distinct from the canonical height $\hh_L$ on $G$.

\subsection{A uniform height comparison}\label{sec:aux-limit}

The following proposition compares the height of an auxiliary line bundle with a weighted sum of projective heights. It is the function-field analogue of \cite[Theorem~1.2]{Rem05} with $\omega=0$; \cite[Theorem~1.1]{Dil20} treats different projective factors over $\overline{\bQ}$. All heights use the normalization of Section~\ref{sec:aux-vector}.
    \begin{proposition}\label{prop:geometric-comparison}
        Let $X\subset\mathbb P^N$ be an integral projective variety of positive dimension, and let $L_X=\mathcal O_X(1)$. Fix integers $m\geq2$, $t_1,t_2,r_0,\theta\geq1$ and a real number $\delta\geq1$. For a tuple $a=(a_1,\ldots,a_m)$ of positive integers, put $|a|=\sum_i a_i$. Let $\pi:\mathcal X\to X^m$ be projective and birational, let $\mathcal P$ be very ample on $\mathcal X$, and put $\mathcal N_a=\pi^*(\boxtimes_i L_X^{a_i})$. Let $\mathcal M$ be a line bundle on $\mathcal X$, and let $\Sigma$ be a family of $r_0$ global sections of $\mathcal P\otimes\mathcal M^{-1}$. Suppose there are injections
        \[
            \mathcal P\longrightarrow\mathcal N_a^{t_1},\qquad
            \mathcal P\otimes\mathcal M^{-1}\longrightarrow\mathcal N_a^{t_2}.
        \]
        Let $U$ be an open set on which $\pi$ and both injections are isomorphisms. Identify $U$ with its image under $\pi$, and fix $x=(x_1,\ldots,x_m)\in U$. Assume the following.
        \begin{enumerate}[beginpenalty=10000]
            \item The line bundle $\mathcal M$ is nef, and $\Sigma$ generates $\mathcal P\otimes\mathcal M^{-1}$.

            \item In an embedding given by $\mathcal P$, the first injection sends the coordinate sections to monomials with coefficient $1$. The second sends $\Sigma$ to polynomials $P_1,\ldots,P_{r_0}$ of multidegree $t_2a$, whose joint coefficient vector satisfies
            \[
                h(P_1,\ldots,P_{r_0})\leq\delta|a|.
            \]

            \item For every product $Y=\prod_iY_i\subset X^m$ containing $x$, with each $Y_i$ integral, its strict transform $\widetilde Y$ satisfies
            \begin{equation}\label{eq:aux-intersection-hypothesis}
                (\mathcal M^{\dim Y}\cdot\widetilde Y)
                \geq\theta^{-1}\prod_i a_i^{\dim Y_i}.
            \end{equation}
        \end{enumerate}
        Define
        \[
            h_{\mathcal M}(x)=h_{\mathcal P}(x)-h(\Sigma(x)),
        \]
        where $h_{\mathcal P}$ uses the stated coordinate sections. There are constants $c_1,c_2,c_3>0$, depending only on $X,N,m,t_1,t_2,r_0,\theta,\delta$, such that
        \begin{equation}\label{eq:aux-generalized-conclusion}
            \sum_i a_i h(x_i)\leq c_1h_{\mathcal M}(x)
        \end{equation}
        whenever $a_i/a_{i+1}\geq c_2$ for $1\leq i<m$ and $h(x_i)\geq c_3$ for all $i$. The constants are uniform in $\mathcal X$, the weights, the point, and their fields of definition.
    \end{proposition}

The auxiliary section will belong to $\mathcal M^d\otimes\mathcal N_a^{-d\varepsilon}$, with $\varepsilon>0$ rational and $d$ a large integer. If its first nonzero derivative at $x$ has small weighted order, the product formula gives the height comparison. Otherwise, the index estimate produces a hypersurface through one of the $x_i$, with controlled degree and height. We choose a product through $x$ of least dimension subject to bounds preserved by this construction, so the second alternative is impossible. Section~\ref{sec:comparison-completion} chooses these bounds and completes the proof.

\subsubsection{Finite projections}

Fix the data of Proposition~\ref{prop:geometric-comparison}. We work on a product $Y=\prod_iY_i\subset X^m$ containing $x$, with each $Y_i$ integral, so that the estimates will also apply after replacing a factor by a proper subvariety. Put
\[
    u_i=\dim Y_i,\qquad D_i=\deg Y_i,\qquad u=\sum_i u_i,\qquad D=\prod_iD_i.
\]
We use this notation until Section~\ref{sec:comparison-completion}. Let $\widetilde Y$ be the strict transform of $Y$ in $\mathcal X$, and let $W_{ij}$ be the original homogeneous coordinates on the $i$th factor. We retain these coordinates because the coordinate sections of $\mathcal P$ map to monomials in them. Omit coordinates identically zero on $Y_i$, together with the resulting zero projective coordinates and generating sections. If a remaining coordinate vanishes at $x_i$, it already gives a proper hyperplane section through $x_i$ with coefficient height zero.

Finite linear projections provide local parameters for differentiation wherever they are \'etale. Their equations also control the coefficients of the auxiliary section. The following construction is the geometric-height counterpart of \cite[Proposition~2.2 and Lemma~2.3]{Rem05}; constant coefficients have absolute value at most $1$ at every place.

\begin{samepage}
    \begin{lemma}\label{lem:comparison-projections}
        For every $i$, there are constant linear forms $V_{i0},\ldots,V_{iu_i}$ defining a finite projection $\rho_i:Y_i\to\mathbb P^{u_i}$ of degree $D_i$, with $V_{i0}(x_i)\ne0$, having the following properties.
        \begin{enumerate}[beginpenalty=10000]
            \item Each remaining $W_{ij}$ satisfies a monic homogeneous equation $Q_{ij}(V_{i0},\ldots,V_{iu_i},W_{ij})=0$ of degree $D_i$. If $u_i<N$, the same holds for an additional constant linear form whose ratio to $V_{i0}$ generates the function-field extension. Let $B_i$ be the joint coefficient vector of these equations, including their common leading coefficient $1$. Then
            \begin{equation}\label{eq:aux-projection-height}
                h(B_i)\leq h(Y_i).
            \end{equation}
            \item Write $Q_{ij}=Q_{ij,1}^{b_{ij}}$, where $Q_{ij,1}$ is monic and irreducible. The discriminants of these irreducible equations and of the equation for the additional linear form are nonzero homogeneous polynomials $\Delta$ in the projection coordinates satisfying
            \begin{equation}\label{eq:aux-discriminant}
                \deg\Delta\leq2D_i^2,\qquad h(\Delta)\leq2D_i h(B_i).
            \end{equation}
            If these discriminants do not vanish at $\rho_i(x_i)$, then $\rho_i$ is \'etale at $x_i$ and all coordinate derivative denominators are nonzero there. For $Y_i=\mathbb P^N$, one may take the identity projection and linear coordinate equations.
        \end{enumerate}
    \end{lemma}
\end{samepage}

\begin{proof}
    Generic projection and the primitive element theorem give a nonempty open set of choices for the projection and the additional linear form. We can also require $V_{i0}(x_i)\ne0$. The coefficients can be chosen in $k$, since constant points are Zariski dense in the parameter space.

    Specialize the Chow form of $Y_i$ to the $u_i$ hyperplanes $V_{i\ell}-T_\ell V_{i0}$ and the hyperplane $W_{ij}-ZV_{i0}$. The coefficient of $Z^{D_i}$ is obtained by taking $V_{i0}$ as the last hyperplane; it is a nonzero scalar independent of $j$. Divide by this common scalar and homogenize. This gives $Q_{ij}$, and the same construction applies to the additional linear form. The coefficients are constant linear combinations of the Chow coefficients, up to one common scalar. The ultrametric inequality bounds their local norms, and the product formula removes the scalar, proving \eqref{eq:aux-projection-height}.

    Multiplicativity of the Gauss norm gives $|Q_{ij,1}|_v=|Q_{ij}|_v^{1/b_{ij}}\leq|B_i|_v$, since $B_i$ contains $1$. Discriminants are resultants of total coefficient degree at most $2D_i$, which proves \eqref{eq:aux-discriminant} for the coordinate equations and the equation of the additional linear form. If the discriminant of the latter equation does not vanish at $\rho_i(x_i)$, the localized algebra generated by the ratio of that form to $V_{i0}$ is \'etale over the base, hence normal. The coordinate algebra of $Y_i$ is integral over this algebra and has the same fraction field, so the two algebras agree there. Thus $\rho_i$ is \'etale at $x_i$. Nonvanishing of the other discriminants gives the asserted derivative denominators.
\end{proof}

Choose these projections over a finite extension $F/K$ defining $x$, $Y$, the line bundles, the sections, and the projection equations. Enlarge $F$ to contain the finitely many roots introduced in Section~\ref{sec:local-differentiation}. The field $F$ and the weights $a_i$ remain fixed while the auxiliary degree tends to infinity. Hypersurface intersections may later require further extensions, under which the normalized heights remain unchanged.

\subsection{Existence of an auxiliary section}\label{sec:aux-sections}

Assume $u>0$ and fix an upper bound $D_*\in\mathbb Z$ for $D$. Put
\[
    \varepsilon=\frac{1}{4u\theta(t_1m)^uD_*},\qquad
    \mathcal Q_d=\mathcal M^d\otimes\mathcal N_a^{-d\varepsilon},
\]
where $d$ runs through positive integers with $d\varepsilon\in\mathbb Z$. All line bundles below are restricted to $\widetilde Y$. The intersection hypothesis ensures that subtracting $d\varepsilon\mathcal N_a$ from $d\mathcal M$ leaves a space of sections of order $d^u$. We use the dimension argument of \cite[Proposition~4.1]{Rem05}, with the nef hypothesis required in \cite[proof of Proposition~3.1]{Dil20}.

\begin{samepage}
    \begin{lemma}\label{lem:auxiliary-dimension}
        As $d$ tends to infinity through positive integers with $d\varepsilon\in\mathbb Z$,
        \begin{equation}\label{eq:aux-section-dimension}
            h^0(\widetilde Y,\mathcal Q_d)
            \geq \frac{d^u}{4\theta u!}\prod_i a_i^{u_i}+o(d^u).
        \end{equation}
    \end{lemma}
\end{samepage}

\begin{proof}
    Choose an effective Cartier divisor $E\in|\mathcal N_a|$ on $\widetilde Y$ and put $q=d\varepsilon$. Successively restricting to $E$ gives
    \[
        h^0(d\mathcal M-q\mathcal N_a)
        \geq h^0(d\mathcal M)-q\,h^0(E,d\mathcal M|_E).
    \]
    Here each $(d\mathcal M-j\mathcal N_a)|_E$ injects into $d\mathcal M|_E$: multiply by a section of $j\mathcal N_a|_E$ that avoids the associated points of $E$. Such a section exists because $\mathcal N_a$ is globally generated and the ground field is infinite. By the nef asymptotic Riemann--Roch formula \cite[Theorem~1.4.40 and Corollary~1.4.41]{Laz04}, the two terms on the right have leading coefficients $\mathcal M^u/u!$ and $\mathcal M^{u-1}E/(u-1)!$. On the possibly nonreduced divisor $E$, use the coherent-sheaf form of the same estimate and Riemann--Roch on its fundamental cycle; lower-dimensional terms are $o(d^{u-1})$. Hence
    \[
        \liminf_{d\to\infty}\frac{u!}{d^u}
        h^0(\widetilde Y,\mathcal M^d\mathcal N_a^{-d\varepsilon})
        \geq \mathcal M^u-u\varepsilon\mathcal M^{u-1}\mathcal N_a.
    \]
    For $u=1$, the restriction term is simply the length of $E$. The injection into $\mathcal N_a^{t_1}$ is nonzero on $\widetilde Y$, because it is an isomorphism at $x$. Thus $t_1\mathcal N_a-\mathcal P$ is effective there. The class $\mathcal P-\mathcal M$ is globally generated. Intersecting these classes with the nef classes $\mathcal M$ and $\mathcal N_a$ gives
    \[
        \mathcal M^{u-1}\mathcal N_a
        \leq t_1^{u-1}\mathcal N_a^u
        \leq(t_1m)^u D\prod_i a_i^{u_i}.
    \]
    Together with \eqref{eq:aux-intersection-hypothesis} and the choice of $\varepsilon$, this bounds the liminf below by $3(4\theta)^{-1}\prod_i a_i^{u_i}$. In particular, \eqref{eq:aux-section-dimension} holds.
\end{proof}

As in \cite[Proposition~4.2]{Rem05}, we express a section of $\mathcal Q_d$ by polynomials whose coefficients satisfy linear equations. Set $r=r_0(N+1)^m$ and let $s_\lambda$ denote the $r$ sections
\[
    \sigma^d\prod_i W_{i\ell_i}^{d\varepsilon a_i},
    \qquad \sigma\in\Sigma,\quad0\leq\ell_i\leq N,
\]
which generate $\mathcal P^d\mathcal M^{-d}\mathcal N_a^{d\varepsilon}$. For large $d$, Serre vanishing allows us to choose a basis of $H^0(\widetilde Y,\mathcal P^d)$ consisting of monomials in the projective coordinates. Write a family $\mathbf F=(F_\lambda)_{1\leq\lambda\leq r}$ in this basis. It represents a section of $\mathcal Q_d$ precisely when the local quotients $F_\lambda/s_\lambda$ agree, or equivalently when
\[
    F_\lambda s_\mu=F_\mu s_\lambda\qquad(1\leq\lambda,\mu\leq r).
\]
Let $N_d$ be the number of coefficient unknowns and $q_d$ the dimension of the kernel of this linear system. We allow zero generating sections, so omitting coordinates does not change $r$.

\begin{samepage}
    \begin{lemma}\label{lem:auxiliary-system}
        The compatibility system has the following properties.
        \begin{enumerate}[beginpenalty=10000]
            \item Its kernel is naturally $H^0(\widetilde Y,\mathcal Q_d)$, and
            \begin{equation}\label{eq:aux-rank-ratio}
                \frac{N_d}{q_d}\leq4\theta r(t_1m)^uD+o(1).
            \end{equation}
            \item It can be written as scalar linear equations over $F$ whose joint coefficient vector $\mathcal A_d$ satisfies
            \begin{equation}\label{eq:aux-matrix-height}
                h(\mathcal A_d)\leq d\left(\delta|a|+(t_1+t_2+\varepsilon)\sum_i a_i h(B_i)\right)+O_{Y,a}(1).
            \end{equation}
        \end{enumerate}
    \end{lemma}
\end{samepage}

\begin{proof}
    Since the $s_\lambda$ generate, every compatible family determines a section of $\mathcal Q_d$ by gluing its local quotients. Multiplication by the $s_\lambda$ gives the inverse map. Thus $q_d=h^0(\widetilde Y,\mathcal Q_d)$. The injection into $\mathcal N_a^{dt_1}$ gives
    \[
        N_d\leq r\frac{d^u}{u!}(t_1m)^uD\prod_i a_i^{u_i}+o(d^u),
    \]
    and Lemma~\ref{lem:auxiliary-dimension} proves \eqref{eq:aux-rank-ratio}.

    Since the coordinate sections of $\mathcal P$ map to monomials, each coefficient in a compatibility equation is a polynomial of the form $\pm W^\alpha P_j^d$ of multidegree $\gamma=d(t_1+t_2+\varepsilon)a$. Reduce this polynomial using the monic equations $Q_{ij}$. The result is a sum of monomials $W^k$ with $k_{ij}<D_i$, whose coefficients are polynomials in the projection coordinates. Each substitution decreases the total degree in the non-projection coordinates, so along any branch there are at most $\gamma_i$ substitutions in factor $i$. Since $|B_i|_v\geq1$, the ultrametric inequality bounds the norm of the entire coefficient family by
    \[
        |(P_1,\ldots,P_{r_0})|_v^d\prod_i|B_i|_v^{\gamma_i}.
    \]
    There are only finitely many possible monomials $W^k$, independently of $d$. Express them in the function-field basis given by powers of the primitive elements, with exponent less than $D_i$ in factor $i$. A single nonzero polynomial $R_Y$ in the projection variables clears the denominators of all these expressions; write $R_YW^k=S_k$ in this basis. Multiplying the compatibility equations by $R_Y$ preserves their kernel. Substitution, followed by taking coefficients in the projection variables and powers of the primitive elements, gives scalar equations over $F$. The polynomials $S_k$ are fixed as $d$ varies, so their contribution to the height is bounded independently of $d$. Summing the local bounds proves \eqref{eq:aux-matrix-height}.
\end{proof}

Thunder's function-field Siegel lemma gives a nonzero solution in terms of the row heights. We use \cite[Theorem~1.1 and (14)]{Fuk10} in place of the number-field estimate \cite[Lemma~2.6]{Rem05}. Let $g_F$ be the genus of $C_F$. A system with $N$ unknowns and $s<N$ independent nonzero rows has a nonzero kernel vector $z$ satisfying
\begin{equation}\label{eq:siegel}
    h(z)\leq\frac{\sum_i h(\text{row }i)}{N-s}
    +\frac{2g_F+1}{[F:K]}.
\end{equation}
For this normalization, choose $O\in C_F(k)$ and, by Riemann--Roch, a nonconstant $t\in F$ whose only pole is at $O$, of order $e\leq g_F+1$. Then $[F:k(t)]=e$. Since $k$ is algebraically closed, the constant $m(F,k(t))$ in \cite[(14)]{Fuk10} equals $e$. With heights normalized by $e$, the cited basis bound gives a vector of logarithmic height at most the kernel's Pl\"ucker height divided by $N-s$, plus $(g_F-1+e)/e$. By Pl\"ucker duality \cite[(21)]{Fuk10} and the ultrametric determinant bound, the kernel height is at most the sum of the row heights. Multiplication by $e/[F:K]$ gives an additive term at most $2g_F/[F:K]$, and hence \eqref{eq:siegel}. If $s=0$, a coordinate vector has height zero.

The genus term depends on $F$ but not on $d$. Dividing by $d$ with $F$ fixed gives the following estimate.

\begin{samepage}
    \begin{lemma}\label{lem:auxiliary-small-section}
        For all sufficiently large positive integers $d$ with $d\varepsilon\in\mathbb Z$, there is a nonzero section $s_d\in H^0(\widetilde Y,\mathcal Q_d)$ represented by a family $F_d$ of homogeneous polynomials in the chosen monomial basis such that
        \begin{equation}\label{eq:aux-height-explicit}
            \limsup_{d\to\infty}\frac{h(F_d)}{d}
            \leq4\theta r(t_1m)^uD\left(\delta|a|+(t_1+t_2+\varepsilon)\sum_i a_i h(B_i)\right).
        \end{equation}
        The expression on the right is independent of $F$ and linear in the weights. The lower bound on $d$ may depend on the fixed data, including $F$ and $a$.
    \end{lemma}
\end{samepage}

\begin{proof}
    Apply \eqref{eq:siegel} to independent rows of the compatibility system. Each row has height at most $h(\mathcal A_d)$, so a nonzero kernel vector can be chosen with
    \[
        h(F_d)\leq\frac{N_d}{q_d}h(\mathcal A_d)+\frac{2g_F+1}{[F:K]}.
    \]
    Divide by $d$ and use \eqref{eq:aux-rank-ratio} and \eqref{eq:aux-matrix-height}. All data, including $F$ and $a$, are fixed while $d$ tends to infinity, and hence
    \begin{equation}\label{eq:fieldlimit}
        \lim_{d\to\infty}\frac{2g_F+1}{d[F:K]}=0.
    \end{equation}
    Neither this term nor $O_{Y,a}(1)/d$ contributes to the limit, which proves \eqref{eq:aux-height-explicit}.
\end{proof}

\subsection{A height bound from differentiation}\label{sec:local-differentiation}

We next estimate the contribution of differentiation to the height inequality. The local input is \cite[Lemma~2.7, pp.~473--474]{Rem05}: its proof at a simple root, on pp.~474--475, uses only the Leibniz rule and the ultrametric inequality and applies at every place of $F$. All derivatives are taken over $F$.

Let $T=(T_1,\ldots,T_u)$ be affine coordinates, and let $y(T)$ be a local algebraic function satisfying a polynomial relation
\[
    Q(T,y)=0,\qquad Q_Y(0,y_0)\ne0,
\]
where $y_0=y(0)$ and $Q_Y=\partial Q/\partial Y$. At a fixed place $v$, rescale the coordinate so that $|y_0|_v\leq1$, and let $|Q|_v$ be the maximum of the absolute values of its coefficients after this rescaling. For a multi-index $\alpha$, write $\partial^\alpha=(\alpha!)^{-1}\partial^{|\alpha|}/\partial T^\alpha$. Put
\[
    c_v=\frac{|Q|_v}{|Q_Y(0,y_0)|_v}\geq1.
\]
The local argument at a simple root in the cited proof uses no irreducibility hypothesis on $Q$ and gives
\begin{equation}\label{eq:derivative}
    |\partial^\alpha y(0)|_v\leq c_v^{2|\alpha|-1}
    \qquad(|\alpha|>0).
\end{equation}

Return to the projections of Lemma~\ref{lem:comparison-projections}. Assume that their discriminants and the remaining original coordinates do not vanish at $x$. The projections are then \'etale at $x$. On a neighborhood of $x_i$ in the chart $V_{i0}\ne0$, the remaining affine coordinate functions are invertible. Write
\[
    w_{ij}=W_{ij}/V_{i0},\qquad z_{ij}=w_{ij}(x_i),\qquad
    \mathbf v_i=(V_{i1}/V_{i0},\ldots,V_{iu_i}/V_{i0}).
\]
The entries of $\mathbf v_i-\mathbf v_i(x_i)$ are local coordinates at $x_i$. To bound the coefficients of the coordinate equations after translation to $x_i$, put, for each place $v$ of $F$,
\[
    M_{iv}=\max_j|z_{ij}|_v,\qquad A_{iv}=|B_i|_vM_{iv}^{D_i}.
\]
Since $1=V_{i0}/V_{i0}$ is a constant linear combination of the $w_{ij}$, we have $M_{iv}\geq1$. Every affine projection coordinate has norm at most $M_{iv}$ at $x_i$, and $|B_i|_v\geq1$. With $Q_{ij}=Q_{ij,1}^{b_{ij}}$ as above, define
\[
    R_{ij}=\left(\frac{1}{b_{ij}!}\partial_Y^{b_{ij}}Q_{ij}\right)
    (1,\mathbf v_i(x_i),z_{ij}),\qquad
    c_i=\prod_j(z_{ij}^{b_{ij}}R_{ij})^{1/b_{ij}}.
\]
The discriminant condition gives $R_{ij}\ne0$. The product $c_i$ collects the derivative denominators in factor $i$, which will disappear from the global estimate by the product formula. The roots in its definition lie in the fixed field $F$ and do not depend on $d$. Derivatives below are divided derivatives in the local projection coordinates. The next lemma is the non-archimedean form of \cite[Corollary~5.1 and Lemma~5.2]{Rem05}; see also the correction in \cite[Lemma~5.1]{Dil20}.

\begin{samepage}
    \begin{lemma}\label{lem:laurent-derivatives}
        For every Laurent monomial $J$ in the coordinates $w_{ij}$, every multi-index $\kappa=(\kappa_i)$, and every place $v$ of $F$,
        \begin{equation}\label{eq:aux-laurent-derivative}
            |\partial^\kappa J(x)|_v\leq |J(x)|_v
            \prod_i A_{iv}^{2(N+1)|\kappa_i|}|c_i|_v^{-2|\kappa_i|}.
        \end{equation}
    \end{lemma}
\end{samepage}

\begin{proof}
    Translate the projection coordinates to $x_i$ and replace $Y$ by $z_{ij}Y$ in the coordinate equation. Its coefficient norm is at most $A_{iv}$, and the branch $w_{ij}/z_{ij}$ takes the value $1$ at the origin. Since $Q_{ij}=Q_{ij,1}^{b_{ij}}$, the Gauss norm of the transformed irreducible equation is at most $A_{iv}^{1/b_{ij}}$. Its first derivative in $Y$ at this root, raised to the power $b_{ij}$, equals $z_{ij}^{b_{ij}}R_{ij}$. Applying \eqref{eq:derivative} to this simple root gives, for $|\alpha|>0$ and $\eta=1$,
    \[
        |\partial^\alpha w_{ij}^{\eta}(x_i)|_v
        \leq |z_{ij}^{\eta}|_v
        \left(\frac{A_{iv}^{1/b_{ij}}}{|(z_{ij}^{b_{ij}}R_{ij})^{1/b_{ij}}|_v}\right)^{2|\alpha|}.
    \]
    The ratio is at least $1$, so we may increase the exponent from $2|\alpha|-1$ to $2|\alpha|$. Apply the same argument to the reciprocal polynomial $Y^{D_i}Q_{ij}(1,\mathbf v_i(x_i)+T,z_{ij}/Y)$. Its coefficient norm is unchanged, and its relevant derivative at $Y=1$ differs only by a sign. The displayed estimate therefore also holds for $\eta=-1$.

    The product of these ratios over $j$ is at most $A_{iv}^{N+1}/|c_i|_v$. Write $J$ as a product of coordinates and inverse coordinates and distribute the divided derivatives among its factors. Each term is bounded by the right-hand side of \eqref{eq:aux-laurent-derivative}; the ultrametric inequality gives the same bound for their sum.
\end{proof}

The relevant order of vanishing weights differentiation in factor $i$ by $(dt_1a_i)^{-1}$. For a nonzero section $s_d$ of $\mathcal Q_d$, let $\sigma_d$ be the least value of
\[
    \sum_i\frac{|\kappa_i|}{dt_1a_i}
\]
for which the divided derivative of a local representative of $s_d$ at $x$ is nonzero. This is the weighted index of $s_d$ at $x$. Multiplication by a unit does not change the least weighted degree of a nonzero term, so the definition is independent of the local generator. Let
\[
    H_x=\sum_i a_i h(x_i),\qquad H_B=\sum_i a_i h(B_i).
\]
We apply the product-formula argument of \cite[proof of Proposition~5.3]{Rem05}. A small weighted index allows the differentiation terms involving $H_x$ to be absorbed into the left side of the height inequality.

\begin{samepage}
    \begin{lemma}\label{lem:height-index}
        Let $s_d$ be represented by the polynomial family $F_d$, and let $\kappa$ attain its weighted index $\sigma_d$. Then
        \begin{enumerate}[beginpenalty=10000]
            \item one has
            \begin{equation}\label{eq:aux-small-index-raw}
                d\varepsilon H_x-dh_{\mathcal M}(x)
                \leq h(F_d)+2(N+1)\sum_i|\kappa_i|\bigl(h(B_i)+D_i h(x_i)\bigr);
            \end{equation}
            \item if $\sigma_d\leq\varepsilon/[4t_1(N+1)\max_iD_i]$, then
            \begin{equation}\label{eq:aux-small-index}
                \frac{\varepsilon}{2}H_x
                \leq h_{\mathcal M}(x)+\frac{h(F_d)}{d}+\frac{\varepsilon}{2}H_B.
            \end{equation}
        \end{enumerate}
    \end{lemma}
\end{samepage}

\begin{proof}
    At each place choose a projective coordinate $\Xi_{j_v}$ of $\mathcal P$ of maximal norm, a section $\sigma_{k_v}$ in $\Sigma$ of maximal norm, and original coordinates $W_{i\ell_{vi}}$ of maximal norm. They define a nonvanishing local generator
    \[
        s_v=\Xi_{j_v}^d\sigma_{k_v}^{-d}\prod_iW_{i\ell_{vi}}^{-d\varepsilon a_i}
    \]
    of $\mathcal Q_d$ near $x$. Fix one rational local generator $s_0$ and write $s_v=\alpha_vs_0$ and $s_d=\alpha s_0$. The coordinate heights give
    \[
        dh_{\mathcal M}(x)-d\varepsilon H_x
        =\frac{1}{[F:K]}\sum_v\log|\alpha_v(x)|_v.
    \]
    The representing family gives $\alpha=\alpha_v F_{k_v,\ell_v}(\Xi/\Xi_{j_v})$. All derivatives of smaller weighted order vanish, so
    \[
        \partial^\kappa\alpha(x)=\alpha_v(x)
        \partial^\kappa\bigl(F_{k_v,\ell_v}(\Xi/\Xi_{j_v})\bigr)(x).
    \]
    Apply the product formula to the nonzero left side. Each monomial on the right is a Laurent monomial in the original $w_{ij}$, since the injection for $\mathcal P$ is monomial. Its value at $x$ has norm at most $1$ by the choice of $\Xi_{j_v}$. Lemma~\ref{lem:laurent-derivatives} bounds its derivative. Sum over the places and apply the product formula to each fixed nonzero $c_i$. Since
    \[
        \frac{1}{[F:K]}\sum_v\log M_{iv}=h(x_i),
    \]
    this gives \eqref{eq:aux-small-index-raw}.

    Under the bound in (2), we have $|\kappa_i|\leq d\varepsilon a_i/[4(N+1)D_i]$. Substitution in \eqref{eq:aux-small-index-raw}, followed by division by $d$, proves \eqref{eq:aux-small-index}.
\end{proof}

\subsection{The geometric product estimate}\label{sec:product-estimates}

We use a large weighted index to cut one factor by a proper hypersurface whose degree and height bounds are independent of the auxiliary degree. The algebraic part of R\'emond's product estimate gives the product structure. We supply the geometric height bound by intersection theory over the base curve.

\subsubsection{Intersections with multiplicities}

Let $s\geq1$ and $n_1,\ldots,n_s\geq0$ be integers. Set $\mathcal P=C_F\times\prod_{i=1}^s\mathbb P^{n_i}$, let $\operatorname{pr}_{C_F}$ be its projection to $C_F$, and let $H_i$ be the hyperplane class from its $i$th projective factor. Write $\mathcal P_F$ for the generic fiber, $n=(n_1,\ldots,n_s)$, $|\alpha|=\sum_i\alpha_i$, $\alpha!=\prod_i\alpha_i!$, and $H^\alpha=\prod_i H_i^{\alpha_i}$. For a subvariety $V$ of codimension $c$ in the generic fiber, let $\mathcal V$ be its closure and put
\[
    d_\alpha(V)=\mathcal V\cdot[\text{fiber}]\cdot H^{n-\alpha}\quad(|\alpha|=c),
    \qquad
    h_\alpha(V)=\frac{\mathcal V\cdot H^{n-\alpha}}{[F:K]}\quad(|\alpha|=c-1).
\]
We use only indices with $0\leq\alpha_i\leq n_i$. The numbers $d_\alpha(V)$ are the multidegrees of $V$, and the numbers $h_\alpha(V)$ are the intersection numbers over $C_F$ used to define its heights. All are nonnegative. Multiprojective Chow heights are fixed linear combinations of the $h_\alpha(V)$.

The geometric B\'ezout estimate corresponding to \cite[Proposition~3.2]{Rem01} follows directly from intersection theory. Let $P$ be a nonzero multihomogeneous polynomial of multidegree $(\delta_1,\ldots,\delta_s)$ that does not vanish identically on $V$. Normalizing one coefficient to $1$ and clearing poles gives a divisor of class
\[
    \operatorname{pr}_{C_F}^*D_P+\sum_i\delta_iH_i,
    \qquad \deg D_P=[F:K]h(P).
\]
The closure of the generic intersection cycle is bounded by this effective intersection cycle. Let $e_i$ denote the $i$th coordinate vector. Intersecting with the hyperplane classes gives, for $|\alpha|=c$,
\begin{equation}\label{eq:bezout}
    h_\alpha(V\cdot V(P))
    \leq\sum_i\delta_i h_{\alpha-e_i}(V)+h(P)d_\alpha(V),
\end{equation}
where $h_{\alpha-e_i}(V)=0$ if $\alpha_i=0$. Iterating this estimate bounds successive proper intersections with their intersection multiplicities.

The index gives a lower bound for the Samuel multiplicity of a component of the ideal generated by the derivatives. We must therefore retain that multiplicity in both estimates. The following is the geometric counterpart of \cite[Proposition~3.2]{Rem01}. Write $\delta^\alpha=\prod_i\delta_i^{\alpha_i}$.

\begin{samepage}
    \begin{lemma}\label{lem:product-intersections}
        Let $\mathcal I$ be a nonzero ideal sheaf on $\mathcal P_F$ generated by multihomogeneous polynomials of multidegrees at most $\delta=(\delta_i)$, with coefficient poles bounded by one effective divisor $E$ on $C_F$. Let $T$ be a component of its zero locus of codimension $c\geq1$, and put $\mu=e(\mathcal I_{\eta_T},\mathcal O_{\mathcal P_F,\eta_T})$, the Samuel multiplicity at the generic point $\eta_T$ of $T$. Then
        \begin{enumerate}
            \item for $|\alpha|=c$,
            \begin{equation}\label{eq:product-geometric-degree}
                \mu d_\alpha(T)\leq\frac{c!}{\alpha!}\delta^\alpha;
            \end{equation}
            \item for $|\alpha|=c-1$,
            \begin{equation}\label{eq:product-geometric-height}
                \mu h_\alpha(T)\leq\frac{c!}{\alpha!}\delta^\alpha
                \frac{\deg E}{[F:K]}.
            \end{equation}
        \end{enumerate}
    \end{lemma}
\end{samepage}

\begin{proof}
    Let $\mathcal T$ be the closure of $T$. Multiplying each generator by all monomials that raise its multidegree to $\delta$ preserves the ideal sheaf, since locally one such monomial is a unit. The resulting sections belong to the nef line bundle
    \[
        \mathcal L=\operatorname{pr}_{C_F}^*\mathcal O_{C_F}(E)
        \otimes\mathcal O(\delta_1,\ldots,\delta_s).
    \]
    Let $B$ be the common zero locus of these sections and put $R_0=[\mathcal P]$. Inductively, cut $R_{j-1}$ by a $k$-linear combination of these sections that is nonzero on each of its components. Such a choice exists because no component of $R_{j-1}$ lies in $B$ and $k$ is infinite. Write the intersection cycle as $A_j+R_j$, where $A_j$ consists of the components contained in $B$, with their multiplicities. The identities
    \[
        c_1(\mathcal L)R_{j-1}=A_j+R_j
    \]
    give
    \[
        c_1(\mathcal L)^c[\mathcal P]
        =R_c+\sum_{j=1}^c c_1(\mathcal L)^{c-j}A_j.
    \]
    No component discarded before the $c$th intersection contains $\eta_T$, since $T$ is a component of the generic base locus. The first $c$ sections therefore form a system of parameters in the regular local ring $\mathcal O_{\mathcal P_F,\eta_T}$. Their local intersection multiplicity is their quotient length, which is at least $\mu$ by \cite[Lemma~3.1]{Rem01}. Hence $A_c$ contains $\mu\mathcal T$.

    Intersect the last cycle identity with respectively $[\mathrm{fiber}]H^{n-\alpha}$ and $H^{n-\alpha}$. All the other terms have nonnegative degree, because the cycles are effective and $\mathcal L$, the $H_i$, and the fiber class are nef. Finally,
    \[
        c_1(\mathcal L)^c
        =\left(\sum_i\delta_iH_i\right)^c
        +c\operatorname{pr}_{C_F}^*[E]
        \left(\sum_i\delta_iH_i\right)^{c-1},
    \]
    since $[E]^2=0$. The coefficients of these two terms give \eqref{eq:product-geometric-degree} and \eqref{eq:product-geometric-height}.
\end{proof}

\subsubsection{Components of the index loci}

Let $P$ be a nonzero multihomogeneous polynomial on $\prod_{i=1}^s\mathbb P^{n_i}$ with positive multidegrees $\delta_i$. Its \emph{weighted index} at a point is the minimum of
\[
    \sum_{i=1}^s\frac{|\alpha_i|}{\delta_i}
\]
over the divided derivatives whose values at the point are nonzero, after dehomogenization in an affine chart containing that point. Here $\alpha_i$ is the multi-index in the $i$th factor, and derivatives are divided as in Section~\ref{sec:local-differentiation}. This agrees with the local-section definition of weighted index used there. Let $Z_\epsilon(P)$ be the closed locus where this index exceeds $\epsilon$. By the Euler identities for a multihomogeneous polynomial, it is also cut out by the homogeneous divided derivatives of weighted order at most $\epsilon$, as in \cite[\S1, p.~289]{Rem01}.

A component occurring at two distinct index levels has a large Samuel multiplicity. If the successive polynomial degrees are sufficiently separated, the resulting multidegree inequalities force the component to be a product. We use the multiplicity estimate in \cite[Proposition~3.1]{Rem01} and the multidegree argument in \cite[proof of Proposition~2.1, pp.~294--296]{Rem01}; Lemma~\ref{lem:product-intersections} supplies the geometric height estimates. For $1\leq c\leq\sum_i n_i$, put
\[
    \phi(c)=\max_{|a|=c}\frac{c!}{a!},
\]
where the maximum is over the $s$-component nonnegative multi-indices.

\begin{samepage}
    \begin{lemma}\label{lem:product-common-component}
        Let $T$ be a common geometric irreducible component of $Z_\sigma(P)$ and $Z_{\sigma+\tau}(P)$, where $\sigma\geq0$ and $\tau>0$, and put $c=\operatorname{codim}T$. Suppose that $\delta_i/\delta_{i+1}>1$ for $1\leq i<s$. Then
        \begin{enumerate}
            \item $\phi(c)/\tau^c\geq1$;
            \item if $\delta_i/\delta_{i+1}>\phi(c)/\tau^c$ for $1\leq i<s$, then $T=T_1\times\cdots\times T_s$ for geometrically integral subvarieties $T_i\subset\mathbb P^{n_i}$. Writing $a_i=\operatorname{codim}T_i$, we have
            \begin{equation}\label{eq:product-common-component}
                \prod_i\deg T_i\leq\frac{\phi(c)}{\tau^c},\qquad
                \delta_i h(T_i)\prod_{j\ne i}\deg T_j
                \leq a_i\frac{\phi(c)}{\tau^c}h(P)\quad(a_i>0).
            \end{equation}
        \end{enumerate}
    \end{lemma}
\end{samepage}

\begin{proof}
    Enlarge $F$ so that $T$ and the factors that occur below are geometrically integral over $F$. Normalized heights are unchanged.

    Normalize one coefficient of $P$ to $1$, and let $E$ be the common pole divisor of its coefficients. Then $\deg E=[F:K]h(P)$. Divided differentiation multiplies coefficients by integers, whose valuations are nonnegative. Thus $E$ also bounds the poles of the derivative family, to which Lemma~\ref{lem:product-intersections} applies.

    For $1\leq i\leq s$, let $t_i$ be the codimension of the projection of $T$ to the factors $i,\ldots,s$; put $t_{s+1}=0$ and $b_i=t_i-t_{i+1}$. Let $\mathcal I_\sigma$ be the ideal generated by the derivatives of weighted order at most $\sigma$, and let $\mu$ be its Samuel multiplicity at $\eta_T$. R\'emond's algebraic multiplicity estimate \cite[Proposition~3.1]{Rem01} gives
    \begin{equation}\label{eq:product-samuel}
        \mu\geq\tau^c\prod_i\delta_i^{b_i}.
    \end{equation}
    The derivative hypothesis in that proposition follows from the Leibniz rule: derivatives of $\mathcal I_\sigma$ of weighted order at most $\tau$ belong to the ideal defining $Z_{\sigma+\tau}(P)$, and hence vanish at $\eta_T$. The proposition and its proof in \cite[\S4, pp.~296--297]{Rem01} are algebraic over a characteristic-zero field, so they apply over $F$.

    To obtain the product structure, we follow the multidegree calculation in \cite[proof of Proposition~2.1, pp.~294--296]{Rem01}. Let $p_i$ be the codimension of the projection of $T$ to the first $i$ factors, put $p_0=0$, and set $a_i=p_i-p_{i-1}$. For this multi-index $a$, we have $d_a(T)>0$, as in the cited proof. Moreover $p_i+t_{i+1}\leq c$, because $T$ lies in the product of its two projections. Combining \eqref{eq:product-geometric-degree} and \eqref{eq:product-samuel} gives
    \begin{equation}\label{eq:product-factor-test}
        1\leq d_a(T)
        \leq\frac{c!}{a!\tau^c}
        \prod_{i<s}\left(\frac{\delta_i}{\delta_{i+1}}\right)^%
        {p_i+t_{i+1}-c}.
    \end{equation}
    Since $p_i+t_{i+1}-c\leq0$ and the successive degree ratios exceed $1$, \eqref{eq:product-factor-test} gives $1\leq\phi(c)/\tau^c$. If every successive degree ratio exceeds $\phi(c)/\tau^c$, then \eqref{eq:product-factor-test} forces $p_i+t_{i+1}=c$ for every $i$. The inclusion of $T$ in the product of its two projections is then an equality, by irreducibility and equality of dimensions. Applying this to all cuts gives $T=\prod_iT_i$.

    For this product, $a_i=\operatorname{codim}T_i=b_i$ and
    \[
        d_a(T)=\prod_i\deg T_i,\qquad
        h_{a-e_i}(T)=h(T_i)\prod_{j\ne i}\deg T_j\quad(a_i>0).
    \]
    For the second identity, each factor closure is integral and hence flat over the smooth curve $C_F$. Their fiber product is flat with integral generic fiber. Its local algebras inject into their generic fibers, so it is integral and equals the closure of $T$. The identity follows from the projection formula. Equations \eqref{eq:product-geometric-degree}--\eqref{eq:product-samuel} give \eqref{eq:product-common-component}.
\end{proof}

We now choose index levels so that some two successive loci have a common component. This gives the following geometric version of \cite[Corollary~1.2]{Rem01}, with bounds independent of the field of definition.

\begin{samepage}
    \begin{lemma}\label{lem:geometric-product}
        Suppose $s\geq2$ and $n_i\geq1$, and put $u=\sum_i n_i$. For $\epsilon>0$, set $B_\epsilon=(s/\epsilon)^u$. If
        \begin{equation}\label{eq:productconditions}
            0<\epsilon<s\left(\frac{\log(u+1)}{2u^2}\right)^u,
            \qquad \frac{\delta_i}{\delta_{i+1}}>B_\epsilon\quad(1\leq i<s),
        \end{equation}
        every geometric irreducible component of $Z_\epsilon(P)$ is contained in a proper product $Z_1\times\cdots\times Z_s$ of geometrically integral projective subvarieties satisfying
        \begin{enumerate}[beginpenalty=10000]
            \item $\prod_i\deg Z_i\leq B_\epsilon$;
            \item for each $i$ with $Z_i\subsetneq\mathbb P^{n_i}$,
            \begin{equation}\label{eq:product}
                \delta_i h(Z_i)\leq n_i B_\epsilon h(P).
            \end{equation}
        \end{enumerate}
        The factors may be defined over a finite extension of $K$; their heights use the normalization above.
    \end{lemma}
\end{samepage}

\begin{proof}
    The numerical calculation in \cite[\S2, pp.~292--293]{Rem01} gives, under \eqref{eq:productconditions},
    \[
        \sum_{j=1}^u\phi(j)^{1/j}B_\epsilon^{-1/j}<\epsilon.
    \]
    Choose $\tau_j>\phi(j)^{1/j}B_\epsilon^{-1/j}$ with $\sum_{j=1}^u\tau_j<\epsilon$ and put $\sigma_j=\sum_{r\leq j}\tau_r$, $\sigma_0=0$. Given a component $W$ of $Z_\epsilon(P)$, choose components $W_j$ of $Z_{\sigma_j}(P)$ such that $W\subset W_u\subset\cdots\subset W_0$. Here $W_0$ has codimension one. Choose the first $j$ such that $\operatorname{codim}W_j\leq j$; this exists because $\operatorname{codim}W_u\leq u$. By minimality of $j$ and the inclusion $W_j\subset W_{j-1}$,
    \[
        j\leq\operatorname{codim}W_{j-1}
        \leq\operatorname{codim}W_j\leq j.
    \]
    Thus $W_j=W_{j-1}$ and its codimension is exactly $c=j$. Then $\phi(c)/\tau_j^c<B_\epsilon$. Lemma~\ref{lem:product-common-component} applies to $W_{j-1}=W_j$, since the degree ratios exceed $B_\epsilon>1$. Writing this common component as $\prod_iT_i$, we obtain
    \[
        \prod_i\deg T_i\leq B_\epsilon,\qquad
        \delta_i h(T_i)\leq n_iB_\epsilon h(P).
    \]
    The product is proper because it is a component of an index locus of the nonzero polynomial $P$.
\end{proof}

For one factor $\mathbb P^n$, the same choice of a common component applies to a polynomial of degree $\delta$ with thresholds $0,\beta/n,\ldots,\beta$, where $n\geq1$ and $0<\beta<1$. A component of $Z_\beta(P)$ is contained in a common component $T$ of two successive loci, of some codimension $1\leq c\leq n$. In the local ring at the generic point of $T$, the ideal generated by the derivatives of weighted order at most the lower of the two successive thresholds has Samuel multiplicity at least $(\beta\delta/n)^c$, by \eqref{eq:product-samuel}. Lemma~\ref{lem:product-intersections} therefore gives
\[
    \deg T\leq(n/\beta)^c\leq(n/\beta)^n,\qquad
    \delta h(T)\leq c(n/\beta)^ch(P)
    \leq n(n/\beta)^nh(P).
\]
We use this form after omitting zero-dimensional factors.

\subsubsection{From a large index to a hypersurface}

We return to Proposition~\ref{prop:geometric-comparison}, retaining $Y$, the finite projections $\rho_i$, and $D,u,H_B$ from Sections~\ref{sec:aux-limit}--\ref{sec:local-differentiation}. To apply the product estimate to an auxiliary section, we take a norm to the product of projective spaces. Fix a positive integer $d$ with $d\varepsilon\in\mathbb Z$, a nonzero section $s_d$ of $\mathcal Q_d$, and its representing family $F_d$. Let $\sigma_d$ be its weighted index. Assume that the discriminants and remaining original coordinates are nonzero at $x$, and put $\rho=\prod_i\rho_i$.

\begin{samepage}
    \begin{lemma}\label{lem:auxiliary-norm}
        Let $s_d\ne0$ be represented by $F_d$. Choose a component of $F_d$ whose corresponding generating section $s_\lambda$ is nonzero at $x$, and take its norm under $\rho$. This gives a nonzero polynomial $G_d$ on $\prod_i\mathbb P^{u_i}$ with the following properties.
        \begin{enumerate}[beginpenalty=10000]
            \item Its multidegree is $dt_1Da$, and its index at $\rho(x)$ for these degrees is at least $\sigma_d/D$.
            \item Its coefficient height satisfies
            \begin{equation}\label{eq:aux-norm-height}
                h(G_d)\leq D\left(h(F_d)+dt_1H_B\right).
            \end{equation}
        \end{enumerate}
    \end{lemma}
\end{samepage}

\begin{proof}
    Substitute the monomial expressions for the projective coordinates of $\mathcal P$ into the chosen component. It is a nonzero section of multidegree $dt_1a$ on $Y$. Its norm is regular because the projection is finite and the section is integral on every affine projection chart; the pole bound at infinity gives multidegree $dt_1Da$.

    Extend the Gauss norm in the affine projection variables to a splitting field. The monic equations and $|B_i|_v\geq1$ bound every conjugate of $W_{ij}/V_{i0}$ by $|B_i|_v$. Each conjugate of the chosen component therefore has norm at most $|F_d|_v\prod_i|B_i|_v^{dt_1a_i}$. Taking the product of the $D$ conjugates and summing over the places gives \eqref{eq:aux-norm-height}.

    Let $\alpha$ be the chosen component in a local trivialization. Because $s_\lambda(x)\ne0$, its index for the weights $dt_1a$ equals that of $s_d$. The quotient $N(\alpha)/\alpha$ is a product of integral conjugates and is integral over the base local ring. It lies in the function field of $Y$, and the local ring $\mathcal O_{Y,x}$ is normal because $\rho$ is \'etale at $x$. The quotient is therefore regular at $x$. Its weighted order is nonnegative. The \'etale projection identifies completed local rings, while the weights for $G_d$ are $D$ times the original weights. This gives the index bound.
\end{proof}

The degree of the norm contains a factor $d$, so the product estimate replaces $h(F_d)$ by $h(F_d)/d$. For the auxiliary sections already constructed, this quotient is bounded as $d$ tends to infinity. Put $n=m\dim X$, and let $s$ be the number of positive $u_i$. We omit the zero-dimensional factors when applying Lemma~\ref{lem:geometric-product}.

\begin{samepage}
    \begin{lemma}\label{lem:auxiliary-cut}
        Suppose $0<\beta<1$, $\sigma_d/D>\beta$, and $a_i/a_{i+1}>(m/\beta)^u$ for $1\leq i<m$. If $s\geq2$, suppose also that
        \[
            \beta<s\left(\frac{\log(u+1)}{2u^2}\right)^u.
        \]
        Set $P=((m+n)/\beta)^n$. Then there is an index $i$ with $u_i>0$ and a nonzero homogeneous polynomial $U$ on $\mathbb P^{u_i}$ vanishing at $\rho_i(x_i)$ such that
        \begin{equation}\label{eq:aux-cut-bounds}
            \deg U\leq P,\qquad
            a_i h(U)\leq\frac{u_iP}{t_1}\left(\frac{h(F_d)}{d}+t_1H_B\right).
        \end{equation}
        Its pullback to $Y_i$ defines a proper hypersurface section through $x_i$.
    \end{lemma}
\end{samepage}

\begin{proof}
    Apply Lemma~\ref{lem:geometric-product} to $G_d$ after omitting the zero-dimensional factors. The ratios between the remaining successive weights still exceed $(m/\beta)^u$. If only one factor remains, use the one-factor estimate following that lemma. In both cases the bound $P$ covers the required degree bound, including $(u/\beta)^u$ in the one-factor case. We obtain a proper factor $Z_i\subsetneq\mathbb P^{u_i}$ through $\rho_i(x_i)$ with degree at most $P$ and
    \[
        dt_1Da_i h(Z_i)\leq u_iP h(G_d).
    \]
    Choose a constant linear projection of $Z_i$ to a hypersurface in $\mathbb P^{\dim Z_i+1}$. Its equation is a constant specialization of the Chow form of $Z_i$. Pulling it back gives a nonzero homogeneous polynomial $U$ vanishing on $Z_i$, with $\deg U\leq\deg Z_i$ and $h(U)\leq h(Z_i)$. Combine these bounds with Lemma~\ref{lem:auxiliary-norm} to obtain \eqref{eq:aux-cut-bounds}. Since $\rho_i$ is finite and surjective, the nonzero equation $U$ does not vanish identically on $Y_i$.
\end{proof}

\subsection{Uniform constants and completion of the comparison}\label{sec:comparison-completion}

We prove Proposition~\ref{prop:geometric-comparison} by choosing a product through $x$ of least dimension among those satisfying prescribed degree and height bounds. At dimension $u$, denote these bounds by $D_u$ and $B_u|a|$. The recursion below allows every proper section supplied by the preceding lemmas. Start with
\[
    n=m\dim X,\qquad D_n=(\deg X)^m,\qquad B_n=\max\{1,h(X),\delta\}.
\]
At dimension $u$, use $e_u$ as the value of $\varepsilon$ in $\mathcal Q_d$, $A_u|a|$ as a bound for $h(F_d)/d$, and $s_u$ as the threshold for the weighted index. Define these and the constants for hypersurface sections recursively, for $u=n,n-1,\ldots,1$, by
\begin{align*}
    e_u&=\frac{1}{4u\theta(t_1m)^uD_u},&
    R_u&=8\theta r_0(N+1)^m(t_1m)^uD_u,\\
    A_u&=R_u\bigl(\delta+(t_1+t_2+1)B_u\bigr)+1,&
    s_u&=\frac{e_u}{8t_1(N+1)D_u},\\
    \beta_u&=\min\{s_u/(2D_u),(4n^2)^{-n}\},&
    P_u&=\left\lceil((m+n)/\beta_u)^n\right\rceil,\\
    L_u&=\max\left\{\frac{nP_u}{t_1}(A_u+t_1B_u)+1,
    2D_uB_u,1\right\},&
    T_u&=\max\{P_u,2D_u^2\},\\
    D_{u-1}&=D_uT_u,&
    B_{u-1}&=(T_u+1)B_u+D_uL_u.
\end{align*}
Lemma~\ref{lem:auxiliary-small-section} gives $h(F_d)/d\leq A_u|a|$ for large $d$ when the exponent is $e_u$. For index at most $s_u$, Lemma~\ref{lem:height-index} gives a height bound. For larger index, $\beta_u$ permits Lemma~\ref{lem:auxiliary-cut}, and $T_u,L_u|a|$ bound the degree and weighted height of its equation. The same bounds include the discriminants of Lemma~\ref{lem:comparison-projections}. Choose
\[
    c_2>\max_{1\leq u\leq n}P_u,\qquad
    c_3>\max\left\{B_0,\max_{1\leq u\leq n}\frac{4(A_u+e_uB_u/2)}{e_u}\right\},
    \qquad c_1=\max_{1\leq u\leq n}\frac{4}{e_u}.
\]
These constants depend only on the data listed in Proposition~\ref{prop:geometric-comparison}.

The recursion for $D_{u-1}$ and $B_{u-1}$ preserves the degree and height bounds when one factor is replaced by an irreducible component of a proper hypersurface section.

\begin{samepage}
    \begin{lemma}\label{lem:comparison-descent}
        Suppose a product $Y=\prod_iY_i$ through $x$ of dimension $u>0$ satisfies
        \[
            \prod_i\deg Y_i\leq D_u,\qquad \sum_i a_i h(Y_i)\leq B_u|a|.
        \]
        Suppose an equation $U$ on the original coordinates of one factor, or on its projection coordinates, has $\deg U\leq T_u$ and $a_i h(U)\leq L_u|a|$, vanishes at $x_i$, and does not vanish identically on $Y_i$. Replacing $Y_i$ by an irreducible component through $x_i$ of the corresponding hypersurface section gives a product $Y'$ of dimension $u-1$ satisfying
        \[
            \prod_j\deg Y'_j\leq D_{u-1},\qquad \sum_j a_jh(Y'_j)<B_{u-1}|a|.
        \]
    \end{lemma}
\end{samepage}

\begin{proof}
    Substitution of constant projection forms does not increase the coefficient height or degree of $U$. B\'ezout's degree inequality and \eqref{eq:bezout} give
    \[
        \prod_j\deg Y'_j\leq D_uT_u=D_{u-1},\qquad
        \sum_j a_jh(Y'_j)\leq T_uB_u|a|+D_uL_u|a|<B_{u-1}|a|.
    \]
    The section is proper, so each of its irreducible components has dimension $u_i-1$. Finite extension to define the chosen component preserves the normalized heights.
\end{proof}

\begin{proof}[Proof of Proposition~\ref{prop:geometric-comparison}]
    Choose a product $Y=\prod_iY_i$ through $x$ of least dimension subject to
    \[
        \prod_i\deg Y_i\leq D_u,\qquad
        \sum_i a_i h(Y_i)\leq B_u|a|,\qquad u=\sum_i\dim Y_i.
    \]
    Such a product exists because $X^m$ satisfies the bounds. Its dimension is positive: if $u=0$, then $Y_i=\{x_i\}$ for every $i$, and $c_3|a|\leq\sum_i a_i h(x_i)\leq B_0|a|$, contrary to $c_3>B_0$.

    Choose the projections of Lemma~\ref{lem:comparison-projections}. If an original coordinate vanishes at $x_i$ but not identically on $Y_i$, its equation has degree $1$ and height $0$. If a discriminant vanishes at $\rho_i(x_i)$, its equation has degree at most $T_u$ and weighted height at most $L_u|a|$, by \eqref{eq:aux-discriminant}. Each equation cuts $Y_i$ properly, the latter because $\rho_i$ is finite and surjective. Lemma~\ref{lem:comparison-descent} contradicts the choice of $Y$ in either case. Thus the projections are \'etale at $x$ and all remaining coordinates and derivative denominators are nonzero there.

    Fix a finite extension $F/K$ defining the point, product, projections, and auxiliary data, as in Section~\ref{sec:aux-limit}. Apply Lemma~\ref{lem:auxiliary-small-section} with $\varepsilon=e_u$ and $D_*=D_u$. For every sufficiently large positive integer $d$ with $de_u\in\mathbb Z$, its section satisfies $h(F_d)/d\leq A_u|a|$. If $\sigma_d>s_u$, then
    \[
        \sigma_d/D>s_u/D_u\geq2\beta_u,\qquad
        a_i/a_{i+1}\geq c_2>P_u.
    \]
    Since $P_u\geq(m/\beta_u)^u$, these inequalities give the weight condition in Lemma~\ref{lem:auxiliary-cut}. The choice $\beta_u\leq(4n^2)^{-n}$ also gives its smallness condition: for $2\leq s\leq u\leq n$, we have $(4n^2)^{-n}<s(2u^2)^{-u}<s(\log(u+1)/(2u^2))^u$. This also applies after zero-dimensional factors are omitted. That lemma gives an equation of degree at most $T_u$ and weighted height at most $L_u|a|$. Lemma~\ref{lem:comparison-descent} again contradicts minimality.

    Hence $\sigma_d\leq s_u$ for every sufficiently large positive integer $d$ with $de_u\in\mathbb Z$. Apply Lemma~\ref{lem:height-index} and let $d$ tend to infinity with $F$ and $a$ fixed. Since $H_B\leq B_u|a|$, we obtain
    \[
        \frac{e_u}{2}H_x\leq h_{\mathcal M}(x)+(A_u+e_uB_u/2)|a|.
    \]
    The inequality $H_x\geq c_3|a|$ and the choice of $c_3$ give $(e_u/4)H_x\leq h_{\mathcal M}(x)$, proving \eqref{eq:aux-generalized-conclusion}. Only the lower bound on $d$ depends on the field and the weights; the constants $c_1,c_2,c_3$ do not.
\end{proof}

\subsection{The semi-abelian difference maps}\label{sec:mixed-completion}

We apply Proposition~\ref{prop:geometric-comparison} to differences on $G$ and its abelian quotient. Since the boundary height has degree one and the abelian height has degree two under multiplication, we use weights $a_i^2$ for the former and $a_i$ for the latter. Both contributions then have degree two in the weights.

Let $t$ be the torus rank of $G$. Choose $n\geq1$ so that $N_0=N^{\otimes n}$ gives the projectively normal embedding of \cite[\S2]{Rem03}, and put $c=(t+1)n$. Set $L^*=M+c\overline\pi^*N=M+(t+1)\overline\pi^*N_0$. We first prove Theorem~\ref{thm:vojta} for its canonical height
\[
    \hh_{L^*}=\hh_M+c(\hh_N\circ\pi).
\]
In Theorem~\ref{thm:vojta}\textup{(2)}, we accordingly use $\hh_M$ and $c\hh_N$ for its two components.

Let $a=(a_1,\ldots,a_m)$ be a tuple of positive integers. Each endpoint occurs in one adjacent difference and each interior point in two. We record these multiplicities by $\eta_i$, and define the exponents $b_i$ used in the line bundles below and their sum $H_a$ by
\[
    \eta_1=\eta_m=1,\qquad \eta_i=2\quad(1<i<m),\qquad
    b_i=\eta_i a_i^2,\qquad H_a=|b|=\sum_i b_i.
\]

For $x=(x_1,\ldots,x_m)\in(V\setminus Z_V)(\overline K)^m$, write
\[
    D_a(x)=\sum_{i<m}\bigl(\hh_M(a_i^2x_i-a_{i+1}^2x_{i+1})
    +c\hh_N(\pi(a_ix_i-a_{i+1}x_{i+1}))\bigr).
\]

We first prove the following estimate for positive integer weights and then use it to deduce Theorem~\ref{thm:vojta}. We prove it on a projective graph, where the weighted differences define line bundles. Compare \cite[Lemma~4.2 and proof of Theorem~4.1]{Rem03}.

\begin{samepage}
    \begin{proposition}\label{prop:weighted-comparison}
        There are positive constants $\kappa,h_0,C_4,C_5$, depending only on $G,V$, the chosen embeddings and heights, such that
        \begin{equation}\label{eq:weightedlower}
            \sum_i\eta_i a_i^2\hh_{L^*}(x_i)\leq C_4D_a(x)+C_5H_a
        \end{equation}
        for all $x_i\in(V\setminus Z_V)(\overline K)$ and positive integers $a_i$ satisfying
        \begin{enumerate}[beginpenalty=10000]
            \item $\hh_{L^*}(x_i)\geq h_0$ for every $i$;
            \item $a_i/a_{i+1}\geq\kappa$ for $1\leq i<m$.
        \end{enumerate}
        The constants are independent of the weights and the fields of definition of the points.
    \end{proposition}
\end{samepage}

Fix $x\in(V\setminus Z_V)(\overline K)^m$ and positive integer weights $a$. We use the graph construction of \cite[\S3]{Rem03}. To apply Proposition~\ref{prop:geometric-comparison}, we need an intersection bound and a coefficient bound with constants independent of $a$. We then compare the Weil height of the line bundle defined by these differences with $D_a(x)$.

\subsubsection{Line bundles and intersections}

Write the chosen standard compactification $\overline G$ as the fiber product over $A$ of the bundles $\mathbb P(\mathcal O_A\oplus Q_j)$, where $Q_1,\ldots,Q_t\in\operatorname{Pic}^0(A)$; for $t=0$, it is $A$. Let $\overline V$ be the closure of $V$ in $\overline G$. The boundary line bundle is the fixed $M$. For positive integers $a_1,\ldots,a_m$, the graph compactification in \cite[\S3]{Rem03} extends the map
\[
    (x_i)\longmapsto(a_i^2x_i-a_{i+1}^2x_{i+1})_{i<m}
\]
to a morphism $\overline\beta_a$. The graph projection to $\overline V^m$ is an isomorphism over $V^m$. On the same graph, let $r_a$ be the morphism to $A^{m-1}$ with components $a_i\pi(x_i)-a_{i+1}\pi(x_{i+1})$. Set
\[
    \mathcal M_a=\overline\beta_a^*(M^{\boxtimes(m-1)})
    \otimes r_a^*(N_0^{\boxtimes(m-1)})^{\otimes(t+1)}.
\]
The bundle $\mathcal M_a$ is nef. To see this, pull the tautological quotient on $\mathbb P(\mathcal O_A\oplus Q_j)$ back to a normalized projective curve. It receives a nonzero map from one of two degree-zero line bundles, so it has nonnegative degree. The boundary bundle $\mathcal O(2)\otimes p^*Q_j^{-1}$ is therefore nef, where $p$ denotes projection to $A$. Tensor products and pullbacks preserve nefness, and $N_0$ is ample.

The line bundle $L^*=M+(t+1)\overline\pi^*N_0$ is very ample for this choice of $n$. Pull $\mathcal N_b=\boxtimes_i(L^*)^{\otimes b_i}$ back to the graph. The section maps of \cite[Lemmas~3.2--3.3 and pp.~204--205]{Rem03} give a very ample bundle $\mathcal P_a$ and injections
\[
    \mathcal P_a\longrightarrow\mathcal N_b^{\otimes(t+2)},\qquad
    \mathcal P_a\otimes\mathcal M_a^{-1}
    \longrightarrow\mathcal N_b^{\otimes4(t+1)}.
\]
The first sends coordinate sections to monomials with coefficient $1$. The bundle $\mathcal P_a\otimes\mathcal M_a^{-1}$ is generated by $r_0=(N_2+1)^m(N_3+1)$ sections, where $N_2,N_3$ are the fixed projective dimensions in the boundary and abelian graph embeddings of \cite[\S3]{Rem03}. For the second injection, choose a coordinate section nonzero at $x$. Both injections are then isomorphisms near $x$. These constructions use the theorem of the cube and the projective embeddings, so they apply over our characteristic-zero field. Thus the parameters in Proposition~\ref{prop:geometric-comparison} are
\begin{equation}\label{eq:semi-abelian-parameters}
    t_1=t+2,\qquad t_2=4(t+1),\qquad
    \theta=2^{m^2}.
\end{equation}
The coefficient parameter $\delta$ will be supplied by Lemma~\ref{lem:graph-coefficients}.

Let $Y=\prod_iY_i$ be a product of integral subvarieties with $x_i\in Y_i\subset\overline V$, and put $u_i=\dim Y_i$ and $u=\sum_i u_i$. Write $\widetilde Y_a$ for its strict transform on the graph and $\boldsymbol1=(1,\ldots,1)$. By Vojta's homogeneity theorem \cite[Theorem~5.5]{Voj96}, it suffices to prove positivity when all weights are one. We use the adjacent-pair form in \cite[pp.~205--206]{Rem03}.
\begin{samepage}
    \begin{lemma}\label{lem:graph-intersection}
        For every such product $Y$, we have
        \begin{equation}\label{eq:semi-abelian-intersection-lower}
            (\mathcal M_a^u\cdot\widetilde Y_a)
            \geq\prod_i a_i^{2u_i}
            \geq2^{-m^2}\prod_i b_i^{u_i}.
        \end{equation}
    \end{lemma}
\end{samepage}

\begin{proof}
    The cited homogeneity theorem gives
    \begin{equation}\label{eq:semi-abelian-intersection-homogeneity}
        (\mathcal M_a^u\cdot\widetilde Y_a)
        =\left(\prod_i a_i^{2u_i}\right)
        (\mathcal M_{\boldsymbol1}^u\cdot\widetilde Y_{\boldsymbol1}).
    \end{equation}
    To pass from the cited proof over $\mathbb C$ to our field, descend the graph, line bundles, and product to a finitely generated characteristic-zero subfield and embed it into $\mathbb C$. Intersection numbers are invariant under field extension. It remains to prove positivity for $a_1=\cdots=a_m=1$.

    Set $U_i=Y_i\cap G$. Since $x_i\in U_i\setminus Z_V$, no $U_i$ is contained in $Z_V$. The difference map
    \[
        (y_1,\ldots,y_m)\longmapsto
        (y_1-y_2,\ldots,y_{m-1}-y_m)
    \]
    is therefore generically finite on $\prod_iU_i$ by \cite[p.~206]{Rem03}, which applies the algebraic argument of \cite[Lemma~2.1]{Rem00} with $m=\dim V+1$. When $a_1=\cdots=a_m=1$, $\mathcal M_{\boldsymbol1}$ is the pullback of the ample bundle $(L^*)^{\boxtimes(m-1)}$ by the extension of this map. The projection formula gives $(\mathcal M_{\boldsymbol1}^u\cdot\widetilde Y_{\boldsymbol1})\geq1$, including when $u=0$. Since $u\leq m(m-1)<m^2$ and $\eta_i\leq2$, this proves \eqref{eq:semi-abelian-intersection-lower}.
\end{proof}

\subsubsection{Polynomials representing multiplication}

We use \cite[Proposition~5.2, pp.~126--128]{Rem00} and \cite[Lemma~3.3]{Rem03} for the coefficient estimate. We give the estimate for geometric heights, keeping track of the common scalar in each coefficient family. Put $n_0=h^0(A,N_0)-1$. Let $\xi_0,\ldots,\xi_{n_0}$ be the coordinate sections of $N_0$, and let $Z_0,\ldots,Z_{n_0}$ be the corresponding polynomial variables. By the theorem of the cube, projective normality, and the K\"unneth formula, there are bihomogeneous polynomials $R_{\ell,\ell'}$ of bidegree $(2,2)$ representing the pullbacks of $\xi_\ell\boxtimes\xi_{\ell'}$ by $(y,z)\mapsto(y+z,y-z)$. Fix these polynomials and let $H_{\mathrm{add}}\geq1$ bound the height of their joint coefficient vector.

Multiplying the sections by a power of one coordinate permits a recursion by addition and subtraction. The exponents below make all the required coordinate powers nonnegative, so this recursion gives polynomials whose coefficient heights have quadratic bounds. Put
\[
    f(r)=\left\lfloor\frac{r^2-1}{8}\right\rfloor\qquad(r\geq1).
\]

\begin{samepage}
    \begin{lemma}\label{lem:multiplication-polynomials}
        For every $r\geq1$ and fixed coordinate index $j$, there is a family $\mathbf Q_{r,j}=(Q_{r,i,j})_{0\leq i\leq n_0}$ of homogeneous polynomials such that
        \begin{enumerate}[beginpenalty=10000]
            \item each polynomial has degree $r^2+f(r)$ and represents $\xi_j^{f(r)}[r]^*\xi_i$ under $[r]^*N_0\simeq N_0^{\otimes r^2}$;
            \item the height $\mathcal H_{r,j}$ of the family's joint coefficient vector satisfies $\mathcal H_{r,j}\leq(r^2-1)H_{\mathrm{add}}$.
        \end{enumerate}
        The estimate is uniform in $j$.
    \end{lemma}
\end{samepage}

\begin{proof}
    Start with $Q_{1,i,j}=Z_i$. For $r>1$, put $r_-=\lfloor r/2\rfloor$ and $r_+=\lceil r/2\rceil$. The addition and subtraction of $[r_+]y$ and $[r_-]y$ give $[r]y$ and $[r_+-r_-]y$. If $r$ is odd, the substituted polynomial $R_{i,j}(\mathbf Q_{r_+,j},\mathbf Q_{r_-,j})$ therefore represents
    \[
        \xi_j^{2f(r_-)+2f(r_+)+1}[r]^*\xi_i.
    \]
    Multiply it by $Z_j^{f(r)-2f(r_-)-2f(r_+)-1}$. If $r$ is even, choose once and for all an index $j_0$ such that $\xi_{j_0}(0)\ne0$. Substitute into $R_{i,j_0}$, divide by the resulting common nonzero scalar, and multiply by $Z_j^{f(r)-4f(r_-)}$. These exponents are nonnegative because
    \[
        f(2s)\geq4f(s),\qquad f(2s+1)\geq2f(s)+2f(s+1)+1.
    \]
    The two constructions have the required degree and represent the required sections. Multiplication by a monomial and division by a common nonzero scalar preserve the height of the coefficient vector. The ultrametric inequality gives
    \[
        \mathcal H_{1,j}=0,\qquad
        \mathcal H_{r,j}\leq H_{\mathrm{add}}+2\mathcal H_{r_-,j}+2\mathcal H_{r_+,j}
        \leq(r^2-1)H_{\mathrm{add}},
    \]
    where the last inequality follows by induction from $1+2(r_-^2-1)+2(r_+^2-1)\leq r^2-1$. This bound holds for each fixed $j$, uniformly in the choice of $j$.
\end{proof}

\subsubsection{Coefficients of the section maps}

We apply these polynomials to adjacent pairs. The formula in \cite[Lemma~3.3, p.~202]{Rem03} then gives the coefficient bound for the second injection.

\begin{samepage}
    \begin{lemma}\label{lem:graph-coefficients}
        The second injection and its generating family can be chosen so that their images are represented by a polynomial family $\mathbf P_a$ of multidegree $4(t+1)b$ satisfying
        \begin{equation}\label{eq:semi-abelian-coefficients}
            h(\mathbf P_a)\leq\delta H_a,
            \qquad\delta=2(t+1)H_{\mathrm{add}}.
        \end{equation}
        The constant is independent of the weights and the fields of definition of the points.
    \end{lemma}
\end{samepage}

\begin{proof}
    For positive integers $a,b$ and fixed coordinate indices $j,k$, the family
    \[
        P_{a,b;j,k;\ell,\ell'}=R_{\ell,\ell'}(\mathbf Q_{a,j},\mathbf Q_{b,k})
    \]
    has bidegree $(2(a^2+f(a)),2(b^2+f(b)))$. Its restriction represents the pullbacks of $\xi_\ell\boxtimes\xi_{\ell'}$ by $(y,z)\mapsto(ay+bz,ay-bz)$, multiplied by $p_1^*\xi_j^{2f(a)}p_2^*\xi_k^{2f(b)}$, where $p_1,p_2:A\times A\to A$ are the projections. Its joint coefficient height is at most
    \[
        H_{\mathrm{add}}+2\mathcal H_{a,j}+2\mathcal H_{b,k}
        \leq2(a^2+b^2)H_{\mathrm{add}}.
    \]
    Choose $j_i$ such that $\xi_{j_i}(\pi(x_i))\ne0$ for each $i$, and write $Z_j^{(i)}$ for the coordinate variables on the $i$th factor. The formula in \cite[Lemma~3.3, p.~202]{Rem03} takes the product of these pair polynomials over adjacent factors, raises each to the power $d$, and multiplies by
    \[
        \prod_i\bigl(Z_{j_i}^{(i)}\bigr)^{2d\eta_i(a_i^2-f(a_i))}.
    \]
    The resulting family has exactly multidegree $4d\eta_i a_i^2$ in factor $i$: the adjacent-pair product contributes $2d\eta_i(a_i^2+f(a_i))$. Its joint coefficient height is at most
    \[
        2dH_{\mathrm{add}}\sum_{i<m}(a_i^2+a_{i+1}^2)
        =2dH_{\mathrm{add}}H_a.
    \]
    For the second injection above, take $d=t+1$. The additional boundary factors are monomials with coefficient $1$. This gives \eqref{eq:semi-abelian-coefficients}. The fixed addition family and invariance of normalized heights under finite extension make $\delta$ independent of the weights and the field of definition of the points.
\end{proof}

\subsubsection{Comparison with canonical heights}

Let $h_{\mathcal M_a}$ be the Weil height defined by the graph coordinates. We next compare it with the canonical heights, uniformly in $a$. The following is the geometric form of the section calculation in \cite[Lemma~4.2]{Rem03}.

\begin{samepage}
    \begin{lemma}\label{lem:graph-height}
        There is a constant $C>0$, depending only on the fixed embeddings and height data, such that
        \begin{equation}\label{eq:weightedcomparison}
            h_{\mathcal M_a}(x)\leq
            \sum_{i<m}\bigl(\hh_M(a_i^2x_i-a_{i+1}^2x_{i+1})
            +c\hh_N(\pi(a_ix_i-a_{i+1}x_{i+1}))\bigr)+CH_a
        \end{equation}
        for all the weights and points under consideration, independently of their fields of definition.
    \end{lemma}
\end{samepage}

\begin{proof}
    After a fixed finite extension, choose $P_j\in A(\overline{K})$ with $Q_j\simeq\tau_{P_j}^*N_0\otimes N_0^{-1}$, where $\tau_{P_j}$ denotes translation by $P_j$. For $P\in\{0,\pm P_j\}$, the identity
    \[
        [2]^*\tau_P^*N_0\simeq(\tau_P^*N_0)^{\otimes2}\otimes N_0^{\otimes2}
    \]
    compares two fixed generating families of sections. Their local norm ratios are bounded on a projective model and contribute zero outside finitely many places. Normalization preserves these bounds after finite extension. The graph construction uses the families with total tensor exponent $O(H_a)$, so the resulting errors are $O(H_a)$ with a fixed implied constant.

    Let $\mathcal E_a=\overline\beta_a^*(M^{\boxtimes(m-1)})$ and choose its Weil height by
    \[
        h_{\mathcal E_a}(x)=h_{\mathcal M_a}(x)
        -(t+1)\sum_{i<m}\hh_{N_0}(a_i\pi(x_i)-a_{i+1}\pi(x_{i+1})),
    \]
    and let $B_a(x)$ be the projective height of the boundary coordinates in \cite[Definition~3.2]{Rem03}. The section formulas and the parallelogram law give
    \begin{align*}
        h_{\mathcal E_a}(x)&=B_a(x)-t\sum_i\eta_i a_i^2\hh_{N_0}(\pi(x_i))+O(H_a),\\
        2B_a(x)+2t\sum_i\eta_i a_i^2\hh_{N_0}(\pi(x_i))
        &\leq B_a([2]x)+O(H_a).
    \end{align*}
    Since $\hh_{N_0}([2]y)=4\hh_{N_0}(y)$, these imply $2h_{\mathcal E_a}(x)\leq h_{\mathcal E_a}([2]x)+CH_a$. Iteration yields
    \[
        h_{\mathcal E_a}(x)\leq2^{-r}h_{\mathcal E_a}([2]^rx)+(1-2^{-r})CH_a.
    \]
    For fixed $a$, $h_{\mathcal E_a}$ differs by a bounded function from the pullback of a Weil height for $M^{\boxtimes(m-1)}$. Its normalized limit is the sum of linear heights in \eqref{eq:weightedcomparison}; the bounded difference vanishes after multiplication by $2^{-r}$. This proves the estimate with $C$ independent of $a$.
\end{proof}

\begin{proof}[Proof of Proposition~\ref{prop:weighted-comparison}]
    For the embedding defined by $L^*$, its Weil height satisfies $h_{L^*}=\hh_{L^*}+O(1)$ uniformly on algebraic points; see \cite[\S2, p.~195]{Rem03}. Apply Proposition~\ref{prop:geometric-comparison} with the parameters \eqref{eq:semi-abelian-parameters}, the intersection bound \eqref{eq:semi-abelian-intersection-lower}, and the coefficient bound \eqref{eq:semi-abelian-coefficients}. Since $b_i/b_{i+1}\geq a_i^2/(2a_{i+1}^2)$, a fixed lower bound on $a_i/a_{i+1}$ supplies the weight hypothesis. The bounded difference between $h_{L^*}$ and $\hh_{L^*}$ supplies the height hypothesis after increasing $h_0$. Combining the conclusion with Lemma~\ref{lem:graph-height} proves \eqref{eq:weightedlower}; the bounded height differences contribute a constant multiple of $H_a$. All constants depend only on the stated fixed data.
\end{proof}

\subsection{Completion of the Vojta inequality}\label{sec:vojta-completion}

\subsubsection{Integer weights and normalized differences}

\begin{proof}[Proof of Theorem~\ref{thm:vojta} over a curve]
    We approximate the normalized points in Theorem~\ref{thm:vojta} by integer multiples and apply Proposition~\ref{prop:weighted-comparison}, as in \cite[Lemma~4.1 and proof of Theorem~4.1]{Rem03}. Choose $C_1>0$, $C_2>\max\{\kappa,1\}$, and $C_3\geq h_0$ so that
    \[
        C_4(C_1^{-1}+C_1^{-2})<1,\qquad 2C_5/C_3<1.
    \]
    Suppose a tuple outside $Z_V$ satisfies the hypotheses of Theorem~\ref{thm:vojta} for $\hh_{L^*}$ with these constants. Put $H_i=\hh_{L^*}(x_i)$. For each positive integer $\nu$, choose $a_i(\nu)$ nearest to $\nu/\sqrt{H_i}$. These integers are positive for large $\nu$, and
    \[
        \frac{a_i(\nu)}{a_{i+1}(\nu)}\longrightarrow
        \sqrt{\frac{H_{i+1}}{H_i}}\geq C_2>\kappa,
        \qquad
        \frac{a_i(\nu)^2}{\nu^2}\longrightarrow\frac1{H_i}.
    \]
    Thus \eqref{eq:weightedlower} applies. Divide it by $\nu^2$ and let $\nu$ tend to infinity. Homogeneity and continuity on the real span of the fixed tuple identify the limit of $D_{a(\nu)}(x)/\nu^2$ with
    \[
        \sum_{i<m}\left(
        \hh_M\left(\frac{x_i}{H_i}-\frac{x_{i+1}}{H_{i+1}}\right)
        +c\hh_N\left(\frac{\pi(x_i)}{\sqrt{H_i}}-\frac{\pi(x_{i+1})}{\sqrt{H_{i+1}}}\right)
        \right).
    \]
    The inequalities in Theorem~\ref{thm:vojta}\textup{(2)}, applied to $\hh_{L^*}$, bound this by $(m-1)(C_1^{-1}+C_1^{-2})$. Since $\sum_i\eta_i=2(m-1)$ and $H_i\geq C_3$, the limiting inequality gives
    \[
        2\leq C_4(C_1^{-1}+C_1^{-2})+\frac{2C_5}{C_3}<2,
    \]
    a contradiction. Here the auxiliary degree has already tended to infinity with the field and the weights fixed. Thus \eqref{eq:weightedlower} has no genus term before the weights vary with $\nu$.

    It remains to replace $\hh_{L^*}$ by $\hh_L$. Suppose that $x,y$ have positive height and the two expressions in Theorem~\ref{thm:vojta}\textup{(2)}, after taking the square root of the second, are at most $\delta$. Set
    \[
        r_x=\frac{\hh_{L^*}(x)}{\hh_L(x)}\in[1,c],\qquad
        r_y=\frac{\hh_{L^*}(y)}{\hh_L(y)}\in[1,c].
    \]
    The identity $r_x=c-(c-1)\hh_M(x)/\hh_L(x)$ and the reverse triangle inequality give $|r_x-r_y|\leq(c-1)\delta$. On $[1,c]$, the functions $s^{-1}$ and $s^{-1/2}$ have Lipschitz constants $1$ and $1/2$. For $\hh_{L^*}$, whose quadratic part is $c(\hh_N\circ\pi)$, the corresponding first expression is at most $c\delta$ and the square root of the second is at most $\sqrt c(1+(c-1)/2)\delta\leq c^{3/2}\delta$. Moreover,
    \[
        \hh_{L^*}(x)\geq\hh_L(x),\qquad
        \frac{\hh_{L^*}(y)}{\hh_{L^*}(x)}\geq\frac{\hh_L(y)}{c\hh_L(x)}.
    \]
    If $C_1,C_2,C_3$ are the constants for $\hh_{L^*}$, choose $c_1\geq c^{3/2}C_1$, $c_2\geq\sqrt c\,C_2$, and $c_3\geq C_3$. The hypotheses for $\hh_L$ would imply those already excluded for $\hh_{L^*}$. This proves Theorem~\ref{thm:vojta} over a curve.
\end{proof}

\subsubsection{Higher transcendence degree}\label{sec:slicing}

To pass to higher transcendence degree, we choose a generic curve on the fixed polarized model. We need a single curve that preserves normalized heights for every finite extension of $K$. Suppose $b=\dim\mathcal B>1$. Choose $e\geq1$ such that $\mathcal H^{\otimes e}$ is very ample, and let $k'$ be an algebraic closure of the function field of the full parameter space of $(b-1)$-tuples of hyperplanes in this embedding. Let $C/k'$ be their generic complete intersection in $\mathcal B_{k'}$. Since $\mathcal B$ is normal, its singular locus has codimension at least two and is avoided by this generic curve. Bertini's theorem in characteristic zero then makes $C$ smooth; it is projective as a closed subvariety of $\mathcal B_{k'}$. The incidence variety dominates $\mathcal B$, so it gives an injection
\[
    K\hookrightarrow k'(C).
\]
We fix an extension of this injection from $\overline K$ to an algebraic closure of $k'(C)$.

Put $\lambda=e^{b-1}$, and denote the base changes of $M,N,L$ by $M_C,N_C,L_C$. Heights over $k'(C)$ use the normalization of Section~\ref{sec:aux-vector}. The generic curve has the following properties.

\begin{samepage}
    \begin{lemma}\label{lem:generic-curve}
        The extension $K\hookrightarrow k'(C)$ has the following properties.
        \begin{enumerate}[beginpenalty=10000]
            \item For every finite extension $F/K$, the algebra $F\otimes_K k'(C)$ is a field of degree $[F:K]$ over $k'(C)$.
            \item The canonical heights satisfy $\hh_{M_C}=\lambda\hh_M$, $\hh_{N_C}=\lambda\hh_N$, and $\hh_{L_C}=\lambda\hh_L$ on the original algebraic points of $G,A,G$, respectively.
            \item Every point of $(V\setminus Z_V)(\overline K)$ remains outside the Mordell exceptional locus after extension to an algebraic closure of $k'(C)$.
        \end{enumerate}
    \end{lemma}
\end{samepage}

\begin{proof}
    Let $F/K$ be any finite extension and let $\mathcal B_F\to\mathcal B$ be the normalization. The normalization morphism is finite, so every nonempty fiber has dimension zero. Thus \cite[Theorem~1.5]{PS22} applies: while the image has dimension at least two, its bound makes the locus of hyperplanes whose inverse images are not geometrically irreducible a proper subset of the dual projective space. Successive generic hyperplanes therefore have geometrically irreducible inverse images. Generic reducedness follows from Bertini in characteristic zero, giving a geometrically integral curve over the parameter field. Thus, after extending the parameter field to $k'$, the induced curve has function field
    \[
        F\otimes_K k'(C),\qquad
        [F\otimes_K k'(C):k'(C)]=[F:K].
    \]
    The good open subset of the parameter space may depend on $F$, but its generic point belongs to every such open subset. Consequently the single field $k'(C)$ has this property for every finite $F/K$. The same generic complete intersection avoids every fixed codimension-two subset of $\mathcal B_F$ defined over $k$; in particular, modifications needed to extend a given rational map do not change the induced curve.

    We record the height comparison, including its normalization. For a projective point $z$ over $F$, normalize one nonzero coordinate to $1$, and let $D_z$ be the common pole divisor of its coordinates on $\mathcal B_F$. Write $\mathcal H_F$ for the pullback of $\mathcal H$. The generic curve avoids the indeterminacy locus of the rational map defined by $z$, so restricting its coordinates and applying the projection formula gives
    \[
        h_C(z)=\frac{D_z\cdot(e\mathcal H_F)^{b-1}}{[F:K]}
        =\lambda h(z).
    \]
    The same calculation applies to a line bundle on a projective model: extend the point on the closure of its graph and restrict the pulled-back line bundle to the generic curve. It also applies to coefficient and Chow heights, while generic multidegrees are unchanged by extension of the field. These identities hold for every algebraic point and every multiple $[n]x$ with the same generic curve. Passing to the canonical limits gives
    \[
        \hh_{M_C}=\lambda\hh_M,\qquad
        \hh_{N_C}=\lambda\hh_N,\qquad
        \hh_{L_C}=\lambda\hh_L.
    \]

    Original points outside $Z_V$ remain outside the Mordell exceptional locus after this extension. Indeed, any positive-dimensional semi-abelian subgroup over the enlarged algebraically closed field descends to $\overline K$ by \cite[Theorem~2.3.12]{Liu2024chowtrace1motiveslangneron}. If a coset of such a subgroup contains $x\in V(\overline K)$, it is $x+B$ after base change, for a subgroup $B$ over $\overline K$. Its containment in the extended variety $V$ descends faithfully flatly, so $x+B\subset V$ and $x\in Z_V$.
\end{proof}

\begin{proof}[Completion of the proof of Theorem~\ref{thm:vojta}]
    Both expressions in Theorem~\ref{thm:vojta}\textup{(2)} are unchanged by the displayed scaling, by their respective linear and quadratic homogeneities. Height ratios are unchanged as well. If $c'_1,c'_2,c'_3$ are the constants over the curve, take $c_1=c'_1$, $c_2=c'_2$, and $c_3=c'_3/\lambda$. This proves the theorem for the original polarization. The trace and bounded-height arguments of Sections~\ref{sec:heights} and~\ref{sec:bounded} continue to use the original extension $K/k$.
\end{proof}

\section{Proof of the main theorem}\label{sec:large}

After passing to the stabilizer quotient, Theorem~\ref{thm:vojta} bounds heights outside a proper closed subset. Proposition~\ref{prop:bounded} will then give the main theorem.

\subsection{Bounded height outside the cosets}

Finite rank makes the quotients by the kernels of the seminorms $\hh_M$ and $\sqrt{\hh_N\circ\pi}$ finite-dimensional. A finite covering of their unit balls supplies the nearby normalized points required by Theorem~\ref{thm:vojta}. The following is the geometric counterpart of \cite[\S5, Lemma~5.1(1) and p.~210]{Rem03}; see also \cite[Proposition~3.3]{Ge24} in the abelian arithmetic setting.

\begin{proposition}\label{prop:large}
    Let $V\subset G_{\overline{K}}$ be a closed subvariety, and let $\Lambda\subset G(\overline{K})$ be a subgroup of finite rank. There are $\eta>0$ and $R\geq0$ such that
    \[
        \hh_L(x)\leq R\qquad
        \bigl(x\in(V\setminus Z_V)(\overline{K})\cap(\Lambda+B_\eta(G))\bigr).
    \]
\end{proposition}

\begin{proof}
    If $V=\{v\}$, take $R=\hh_L(v)$ and $\eta=1$. Otherwise choose the constants in Theorem~\ref{thm:vojta} with $c_1\geq1$ and $c_2>1$, and again set $\eta=1$. We will obtain a finite partition in which pairs of sufficiently large points satisfy Theorem~\ref{thm:vojta}\textup{(2)}.

    Write $x=\lambda+z$, where $\lambda\in\Lambda$ and $\hh_L(z)<1$. The triangle inequalities for $\hh_M$ and $\sqrt{\hh_N\circ\pi}$ give
    \[
        |\hh_L(x)-\hh_L(\lambda)|\leq1+2\sqrt{\hh_L(\lambda)}.
    \]
    The function $\sqrt{\hh_M}$ is subadditive, as is $\sqrt{\hh_N\circ\pi}$. Applying the Euclidean triangle inequality to this pair of functions proves the triangle inequality for $\sqrt{\hh_L}$. In particular,
    \[
        \bigl|\sqrt{\hh_L(x)}-\sqrt{\hh_L(\lambda)}\bigr|\leq1.
    \]
    If $\hh_L(x)>1$, then $\hh_L(\lambda)>0$, and we may normalize by both heights. Homogeneity and the triangle inequalities give
    \begin{align*}
        \hh_M\left(\frac{x}{\hh_L(x)}-\frac{\lambda}{\hh_L(\lambda)}\right)
        &\leq\frac{1+|\hh_L(x)-\hh_L(\lambda)|}{\hh_L(x)}
        \leq\frac{4+2\sqrt{\hh_L(x)}}{\hh_L(x)},\\
        \sqrt{\hh_N\left(\frac{\pi(x)}{\sqrt{\hh_L(x)}}-\frac{\pi(\lambda)}{\sqrt{\hh_L(\lambda)}}\right)}
        &\leq\frac{1+|\sqrt{\hh_L(x)}-\sqrt{\hh_L(\lambda)}|}{\sqrt{\hh_L(x)}}
        \leq\frac{2}{\sqrt{\hh_L(x)}}.
    \end{align*}
    Both bounds tend to zero as $\hh_L(x)\to\infty$. Choose $R_0>\max\{1,c_3\}$ so that they are at most $1/(4c_1)$ for $\hh_L(x)\geq R_0$, independently of $x$ and its decomposition.

    Take the finite-dimensional normed quotients of $\Lambda\otimes_\bZ\bR$ by the kernels of $\hh_M$ and $\sqrt{\hh_N\circ\pi}$, respectively. The images of $\lambda/\hh_L(\lambda)$ in the first quotient and of $\lambda/\sqrt{\hh_L(\lambda)}$ in the second belong to the corresponding closed unit balls, since $\hh_M(\lambda)\leq\hh_L(\lambda)$ and $\hh_N(\pi(\lambda))\leq\hh_L(\lambda)$. Compactness gives a finite $1/(4c_1)$-net in each ball. For every $x$ with $\hh_L(x)\geq R_0$, fix a decomposition $x=\lambda+z$ and assign it to a pair of net centers within $1/(4c_1)$ of the normalized images of $\lambda$. The pairs of centers define a finite partition.

    If $x$ and $x'$ belong to the same part, the normalized subgroup images differ by at most $1/(2c_1)$ in each norm. Combining this with the preceding error bounds gives
    \begin{align*}
        \hh_M\left(\frac{x}{\hh_L(x)}-\frac{x'}{\hh_L(x')}\right)&\leq c_1^{-1},\\
        \hh_N\left(\frac{\pi(x)}{\sqrt{\hh_L(x)}}-\frac{\pi(x')}{\sqrt{\hh_L(x')}}\right)&\leq c_1^{-2}.
    \end{align*}
    If the heights in $(V\setminus Z_V)(\overline K)\cap(\Lambda+B_1(G))$ were unbounded, one part would have unbounded heights. Successively choose $m=\dim V+1$ points in that part with first height at least $R_0$ and successive height ratios at least $c_2^2$. They satisfy both conditions of Theorem~\ref{thm:vojta}, a contradiction. The asserted bound $R$ follows.
\end{proof}

\subsection{The stabilizer quotient}\label{sec:mainproof}

\begin{proof}[Proof of Theorem~\ref{thm:main}]
    Assume~(1). Let $S=\Stab(X)^0$ and $q:G_{\overline{K}}\to\widetilde G=G_{\overline K}/S$ be the stabilizer quotient. Choose a standard compactification of $\widetilde G$ and a line bundle $\widetilde L=\widetilde M+\overline{\widetilde\pi}^{\,*}\widetilde N$, where $\widetilde M$ is its boundary line bundle, $\widetilde N$ is symmetric and ample on its abelian quotient, and $\overline{\widetilde\pi}$ extends the projection from $\widetilde G$ to its abelian quotient. Define $B_\epsilon(\widetilde G)$ using $\hh_{\widetilde L}$. Height comparison for $q$ shows that $q(X)(\overline{K})\cap(q(\Gamma)+B_\epsilon(\widetilde G))$ is dense in $q(X)$ for every $\epsilon>0$.

    The group $q(\Gamma)$ has finite rank and satisfies $[n]^{-1}q(\Gamma)=q(\Gamma)$ for every $n\geq1$, with inverse images taken in $\widetilde G(\overline K)$. Indeed, if $[n]b=q(\gamma)$ for $\gamma\in\Gamma$, lift $b$ to $x\in G(\overline{K})$. Multiplication by $n$ is surjective on $S$, so some $s\in S(\overline{K})$ satisfies $[n]s=[n]x-\gamma$. Then $[n](x-s)=\gamma$, whence $x-s\in\Gamma$ and $b\in q(\Gamma)$. Applying $q$ to the definition of $\Gamma$ gives $q(\Gamma)\subset q(\Gamma_0)^{\mathrm{div}}$, and the inverse-image identities give the reverse inclusion.

    The quotient data descend to a finite extension of $K$, preserving normalized heights. Since $q^{-1}(q(X))=X$, the image $q(X)$ is closed and has finite stabilizer $\Stab(X)/\Stab(X)^0$. Replace $G,X,\Gamma_0,\Gamma$ by their images and $M,N,L$ by the quotient line bundles $\widetilde M,\widetilde N,\widetilde L$, and use this height to define $\Gamma_\epsilon$. It is now enough to prove~(2) when $X$ has finite stabilizer.

    Abramovich's theorem makes $Z_X$ a proper closed subset when $\dim X>0$; for a point, $Z_X$ is empty. Proposition~\ref{prop:large} gives $\eta>0$ and $R\geq0$ such that every point of $(X\setminus Z_X)(\overline{K})\cap\Gamma_\eta$ has height at most $R$. For $0<\epsilon\leq\eta$, the set $(X\setminus Z_X)(\overline{K})\cap\Gamma_\epsilon$ remains dense and has height at most $R$. The same density of points of height at most $R$ holds for $\epsilon>\eta$ by inclusion. Proposition~\ref{prop:bounded} gives~(2) for the quotient, which is exactly \eqref{eq:special} for the original variety.

    Conversely, assume~(2) and choose $\gamma$ in \eqref{eq:special}. The connected stabilizer of $X-\gamma$ is again $S$, and its quotient is the image of a constant subvariety under a homomorphism with finite kernel. Hence $X-\gamma$ is special in the sense of \cite[Theorem~1.1]{LY2025Bogomolov}, with torsion translation zero. That theorem gives a dense set of points of arbitrarily small height on $X-\gamma$. Translation by $\gamma\in\Gamma$ proves~(1).
\end{proof}

\begin{proof}[Proof of Corollary~\ref{cor:exceptional}]
    Induct on $\dim X$. If $X$ is $\Gamma$-special, take $Z_1=X$. Otherwise Theorem~\ref{thm:main} gives $\epsilon>0$ such that the closure of $X(\overline{K})\cap\Gamma_\epsilon$ is proper. An empty closure requires no exceptional subvarieties; this includes the dimension-zero case. Otherwise apply induction to its finitely many irreducible components and take the minimum of $\epsilon$ and their positive thresholds. Every point in the resulting neighborhood lies in a component and hence in one of its exceptional subvarieties.
\end{proof}

\subsection{Positive characteristic}\label{sec:charp}

In this subsection, let $k$ have characteristic $p>0$, with the other field and height conventions unchanged. Stabilizers are given their reduced scheme structure, as in \cite[\S1.1]{LY2025Bogomolov}. The implication from $\Gamma$-specialness to the density of $X(\overline{K})\cap\Gamma_\epsilon$ for every $\epsilon>0$ still follows from \cite[Theorem~1.1]{LY2025Bogomolov}. When $\Gamma_0$ has rank zero, this theorem also gives the converse: $\Gamma=\Gamma_0^{\mathrm{div}}$ is the torsion subgroup, and adding torsion preserves canonical heights.

Lemma~\ref{lem:gap} remains valid with the $F/k$-trace, since the Poincar\'e model and integral boundary degrees do not require separability. Lemma~\ref{lem:rational} and Corollary~\ref{cor:point} then apply unchanged. Together with geometric Bogomolov for a point, they prove Theorem~\ref{thm:main} when $\dim X=0$. The constant term in a point of height zero is expressed through the trace map, without identifying its image with a constant subgroup.

For a split torus, Proposition~\ref{prop:bounded} also holds in positive characteristic. The torus is constant, the incidence estimates hold in every characteristic, and generic finiteness suffices for the argument using a big line bundle in Lemma~\ref{lem:inverse}. Homomorphisms of split tori are given by integer matrices on their character groups, so subtori and homomorphisms between them are defined over $k$. The arguments for descent and for the translating point therefore do not require the characteristic-zero identification of a general semi-abelian trace image.

Proposition~\ref{prop:large} and Theorem~\ref{thm:vojta} fail in positive characteristic, as the classical Frobenius solutions of $u+v=1$ show; see \cite[p.~1046, (1.3)--(1.4)]{DM12}.

\begin{example}\label{ex:frobenius}
    Let $K=k(t)$ and $G=\Gm^2$. On $\overline G=(\mathbb P^1)^2$, take $L=M=\mathcal O(2,2)$ and $N=0$ on the trivial abelian quotient. Thus $\hh_L(u,v)=2h_{\mathcal O(1)}(u)+2h_{\mathcal O(1)}(v)$, with the heights computed over $\mathbb P^1_k$ using $\mathcal O(1)$ as the base polarization. Set
    \[
        X=\{(u,v):u+v=1\}\subset G,\qquad P=(t,1-t),
        \qquad\Gamma=\langle P\rangle^{\mathrm{div}}.
    \]
    A one-dimensional torus coset contained in $X$ would equal $X$. Its coordinates would have the form $az^r,bz^s$ for $a,b\in\overline{K}^*$ and integers $r,s$ not both zero, whose sum cannot be identically $1$. Hence $Z_X$ is empty and the stabilizer is finite. Frobenius gives
    \[
        [p^n]P=(t^{p^n},1-t^{p^n})\in X\cap\Gamma,
        \qquad \hh_L([p^n]P)=p^n\hh_L(P)\longrightarrow\infty.
    \]
    This disproves the positive-characteristic analogue of Proposition~\ref{prop:large}. It also disproves that of Theorem~\ref{thm:vojta}, since both expressions in Theorem~\ref{thm:vojta}\textup{(2)} vanish for any two of these points. The curve is constant and therefore $\Gamma$-special, so it does not contradict the specialness conclusion.
\end{example}

The forward implication of Theorem~\ref{thm:main} for general semi-abelian varieties in positive characteristic is not established by these arguments.


\begin{thebibliography}{CGHX21}
\bibitem[Abr94]{Abr94}
D.~Abramovich.
\newblock Subvarieties of semi-abelian varieties.
\newblock {\em Compositio Math.}, 90(1):37--52, 1994.

\bibitem[BBP18]{BBP18}
F.~Benoist, E.~Bouscaren, and A.~Pillay.
\newblock On function field Mordell--Lang: the semi-abelian case and the socle theorem.
\newblock {\em Proc. Lond. Math. Soc. (3)}, 116(1):182--208, 2018.

\bibitem[CGHX21]{CGHX21}
S.~Cantat, Z.~Gao, P.~Habegger, and J.~Xie.
\newblock The geometric Bogomolov conjecture.
\newblock {\em Duke Math. J.}, 170(2):247--277, 2021.

\bibitem[Con06]{Con06}
B.~Conrad.
\newblock Chow's $K/k$-image and $K/k$-trace, and the Lang--N\'eron theorem.
\newblock {\em Enseign. Math. (2)}, 52(1--2):37--108, 2006.

\bibitem[DM12]{DM12}
H.~Derksen and D.~Masser.
\newblock Linear equations over multiplicative groups, recurrences, and mixing I.
\newblock {\em Proc. Lond. Math. Soc. (3)}, 104(5):1045--1083, 2012.

\bibitem[Dil20]{Dil20}
G.~A.~Dill.
\newblock Generalized Vojta--R\'emond inequality.
\newblock {\em Int. J. Number Theory}, 16(1):107--120, 2020.

\bibitem[Fal83]{Fal83}
G.~Faltings.
\newblock Endlichkeitss\"atze f\"ur abelsche Variet\"aten \"uber Zahlk\"orpern.
\newblock {\em Invent. Math.}, 73(3):349--366, 1983.
\newblock Erratum: {\em Invent. Math.}, 75(2):381, 1984.

\bibitem[Fal91]{Fal91}
G.~Faltings.
\newblock Diophantine approximation on abelian varieties.
\newblock {\em Ann. of Math. (2)}, 133(3):549--576, 1991.

\bibitem[Fal94]{Fal94}
G.~Faltings.
\newblock The general case of S.~Lang's conjecture.
\newblock In {\em Barsotti Symposium in Algebraic Geometry} (Abano Terme, 1991), Perspectives in Mathematics, vol.~15, pp.~175--182. Academic Press, San Diego, CA, 1994.

\bibitem[Fuk10]{Fuk10}
L.~Fukshansky.
\newblock Algebraic points of small height missing a union of varieties.
\newblock {\em J. Number Theory}, 130(10):2099--2118, 2010.

\bibitem[Ge24]{Ge24}
T.~Ge.
\newblock Uniform Mordell--Lang plus Bogomolov.
\newblock {\em Int. Math. Res. Not. IMRN}, 2024(9):7360--7378, 2024.

\bibitem[GH19]{GH19}
Z.~Gao and P.~Habegger.
\newblock Heights in families of abelian varieties and the Geometric Bogomolov Conjecture.
\newblock {\em Ann. of Math. (2)}, 189(2):527--604, 2019.

\bibitem[Gra65]{Gra65}
H.~Grauert.
\newblock Mordells Vermutung \"uber rationale Punkte auf algebraischen Kurven und Funktionenk\"orper.
\newblock {\em Publ. Math. Inst. Hautes \'Etudes Sci.}, 25:131--149, 1965.

\bibitem[Gub07]{Gub07}
W.~Gubler.
\newblock The Bogomolov conjecture for totally degenerate abelian varieties.
\newblock {\em Invent. Math.}, 169(2):377--400, 2007.

\bibitem[HLY26]{HLY26}
Z.~Han, W.~Luo, and J.~Yu.
\newblock Uniform Mordell--Lang conjecture for semi-abelian varieties.
\newblock {\em Preprint}, \href{https://arxiv.org/abs/2609.17233}{arXiv:2609.17233}, 2026.

\bibitem[Hru96]{Hru96}
E.~Hrushovski.
\newblock The Mordell--Lang conjecture for function fields.
\newblock {\em J. Amer. Math. Soc.}, 9(3):667--690, 1996.

\bibitem[Hul26]{Hul26}
N.~Hultberg.
\newblock New gap principle for semi-abelian varieties using globally valued fields.
\newblock {\em Preprint}, \href{https://arxiv.org/abs/2601.04972}{arXiv:2601.04972}, 2026.

\bibitem[Lan60]{Lan60}
S.~Lang.
\newblock Integral points on curves.
\newblock {\em Publ. Math. Inst. Hautes \'Etudes Sci.}, 6:27--43, 1960.

\bibitem[Laz04]{Laz04}
R.~Lazarsfeld.
\newblock Positivity in algebraic geometry I.
\newblock Ergebnisse der Mathematik und ihrer Grenzgebiete, vol.~48. Springer-Verlag, Berlin, 2004.

\bibitem[Liu24]{Liu2024chowtrace1motiveslangneron}
Long Liu.
\newblock {Chow trace of 1-motives and the Lang-N\'eron groups}.
\newblock {\em Preprint}, \href{https://arxiv.org/abs/2404.13200}{arXiv:2404.13200}, 2024.
\newblock Accepted by {\em Algebra \& Number Theory}.

\bibitem[Liu26]{Liu2026pushforward}
Long Liu.
\newblock {Representability of smooth-site direct images and Chow trace of semi-abelian varieties}.
\newblock {\em In preparation}.

\bibitem[LY25]{LY2025Bogomolov}
Wenbin Luo and Jiawei Yu.
\newblock {Geometric Bogomolov conjecture for semiabelian varieties}.
\newblock {\em Preprint}, \href{https://arxiv.org/abs/2505.07193}{arXiv:2505.07193}, 2025.

\bibitem[Man63]{Man63}
Yu.~I.~Manin.
\newblock Rational points on algebraic curves over function fields.
\newblock {\em Izv. Akad. Nauk SSSR Ser. Mat.}, 27(6):1395--1440, 1963.
\newblock Correction: {\em Letter to the Editor}, {\em Math. USSR-Izv.}, 34(2):465--466, 1990.

\bibitem[McQ95]{McQ95}
M.~McQuillan.
\newblock Division points on semi-abelian varieties.
\newblock {\em Invent. Math.}, 120(1):143--159, 1995.

\bibitem[Mor01]{Mor01}
A.~Moriwaki.
\newblock A generalization of conjectures of Bogomolov and Lang over finitely generated fields.
\newblock {\em Duke Math. J.}, 107(1):85--102, 2001.

\bibitem[Poo99]{Poo99}
B.~Poonen.
\newblock Mordell--Lang plus Bogomolov.
\newblock {\em Invent. Math.}, 137(2):413--425, 1999.

\bibitem[PS22]{PS22}
B.~Poonen and K.~Slavov.
\newblock The exceptional locus in the Bertini irreducibility theorem for a morphism.
\newblock {\em Int. Math. Res. Not. IMRN}, 2022(6):4503--4513, 2022.

\bibitem[Rem00]{Rem00}
G.~R\'emond.
\newblock In\'egalit\'e de Vojta en dimension sup\'erieure.
\newblock {\em Ann. Scuola Norm. Sup. Pisa Cl. Sci. (4)}, 29(1):101--151, 2000.

\bibitem[Rem01]{Rem01}
G.~R\'emond.
\newblock Sur le th\'eor\`eme du produit.
\newblock {\em J. Th\'eor. Nombres Bordeaux}, 13(1):287--302, 2001.

\bibitem[Rem03]{Rem03}
G.~R\'emond.
\newblock Approximation diophantienne sur les vari\'et\'es semi-ab\'eliennes.
\newblock {\em Ann. Sci. \'Ecole Norm. Sup. (4)}, 36(2):191--212, 2003.

\bibitem[Rem05]{Rem05}
G.~R\'emond.
\newblock In\'egalit\'e de Vojta g\'en\'eralis\'ee.
\newblock {\em Bull. Soc. Math. France}, 133(4):459--495, 2005.

\bibitem[Sil11]{Sil11}
J.~H.~Silverman.
\newblock Height estimates for equidimensional dominant rational maps.
\newblock {\em J. Ramanujan Math. Soc.}, 26(2):145--163, 2011.

\bibitem[Ull98]{Ull98}
E.~Ullmo.
\newblock Positivit\'e et discr\'etion des points alg\'ebriques des courbes.
\newblock {\em Ann. of Math. (2)}, 147(1):167--179, 1998.

\bibitem[Voj96]{Voj96}
P.~Vojta.
\newblock Integral points on subvarieties of semi-abelian varieties, I.
\newblock {\em Invent. Math.}, 126(1):133--181, 1996.

\bibitem[XY22]{XY22}
J.~Xie and X.~Yuan.
\newblock Geometric Bogomolov conjecture in arbitrary characteristics.
\newblock {\em Invent. Math.}, 229(2):607--637, 2022.

\bibitem[Yam18]{Yam18}
K.~Yamaki.
\newblock Trace of abelian varieties over function fields and the geometric Bogomolov conjecture.
\newblock {\em J. Reine Angew. Math.}, 741:133--159, 2018.

\bibitem[Yua25]{Yua25}
X.~Yuan.
\newblock On Vojta's proof of the Mordell conjecture.
\newblock {\em Preprint}, \href{https://arxiv.org/abs/2508.11888v2}{arXiv:2508.11888}, 2025.

\bibitem[Zha98]{Zha98}
S.-W.~Zhang.
\newblock Equidistribution of small points on abelian varieties.
\newblock {\em Ann. of Math. (2)}, 147(1):159--165, 1998.

\bibitem[Zha00]{Zha00}
S.-W.~Zhang.
\newblock Distribution of almost division points.
\newblock {\em Duke Math. J.}, 103(1):39--46, 2000.
\end{thebibliography}
\end{document}